\documentclass[12pt]{article}
\usepackage{color}
\usepackage{lscape}
\usepackage{afterpage}
\usepackage{pstricks}
\usepackage{pst-plot}
\usepackage{longtable}
\usepackage{dcolumn}
\usepackage{pst-node}
\usepackage[dvips]{graphicx}
\usepackage{amsmath}
\usepackage{amssymb}
\usepackage{amsthm}
\usepackage{multirow}
\usepackage{colortab}
\usepackage{color}
\usepackage{array}
\usepackage{slashbox}
\usepackage{colortbl}
\usepackage{shadbox}
\usepackage{textcomp}
\usepackage{pstricks, pst-node}
\usepackage{pst-all}
\psset{arrows=->, labelsep=3pt, mnode=circle}

\usepackage{eurosym}

\newfont{\rams}{msbm10 scaled\magstep1}
\newcommand{\rea}{\mbox{\rams \symbol{'122}}}

\newenvironment{resumeT}{\begin{list}{}{\setlength{\rightmargin}{\leftmargin}}\item[]
{\centering {\bf \it~~~}
\par}\item[]\ignorespaces}{\unskip\end{list}}
\newtheorem{defn}{Definition}[section]

\newtheorem{prop}{Proposition}[section]

\newtheorem{note}{Note}[section]

\begin{document}
%%%%%%%%%%%%%%%%%%%%
%%%%%%% Front page %%%%%%%
%%%%%%%%%%%%%%%%%%%%
\title{Robust Ordinal Regression in case of Imprecise Evaluations}
\author{ \hspace{0,1cm} Salvatore Corrente\footnote{Department of Economics and Business, University of Catania, Corso Italia, 55, 95129  Catania, Italy, e-mail: \mbox{salvatore.corrente\string@unict.it}}, 
\hspace{0,1cm} Salvatore Greco\footnote{Department of Economics and Business, University of Catania, Corso Italia, 55, 95129  Catania, Italy, e-mail: salgreco\string@unict.it}, 
\hspace{0,1cm} Roman S{\l}owi\'{n}ski\footnote{Institute of Computing Science, Pozna\'{n} University of Technology, 60-965 Pozna\'{n}, and Systems Research Institute, Polish Academy of Sciences, 01-447 Warsaw, Poland, e-mail: roman.slowinski\string@cs.put.poznan.pl }} 

\date{}
\maketitle \vfill
\newpage

%%%%%%%%%%%%%%%%%%%%
%%%%%% Table of contents %%%%%
%%%%%%%%%%%%%%%%%%%%
% \makeatother \addcontentsline{toc}{section}{Table of Contents}
 %\tableofcontents \vfill\newpage

%%%%%%%%%%%%%%%%%%%%
%%%%%%% Abstract %%%%%%%%
%%%%%%%%%%%%%%%%%%%%
\addcontentsline{toc}{section}{Abstract}
\begin{center}
  {\Large
    \textbf{Robust Ordinal Regression in case of Imprecise Evaluations}
  }
\end{center}

\begin{center}
{\large {\bf Abstract}}
\end{center}

\vspace{-1cm}

\begin{resumeT}

Robust Ordinal Regression (ROR) is a way of dealing with Multiple Criteria Decision Aiding (MCDA), by considering all sets of parameters of an assumed preference model, that are compatible with preference information given by the Decision Maker (DM). As a result of ROR, one gets necessary and possible preference relations in the set of alternatives, which hold for all compatible sets of parameters or for at least one compatible set of parameters, respectively. In this paper, we extend the MCDA methods based on ROR, by considering one important aspect of decision problems: imprecise evaluations. To deal with imprecise evaluations of some alternatives on particular criteria, we extend the set of considered variables to define necessary and possible preference relations taking into account this imprecision. In consequence, the concepts of necessary and possible preference represent not only all compatible sets of parameters, but also all possible values of the imprecise evaluations of alternatives. 

\vspace{1cm}
%%%%%%%%%%%%%%%%%%%%
%%%%%%% Keywords %%%%%%%
%%%%%%%%%%%%%%%%%%%%
\end{resumeT}

\vfill\newpage
 \pagenumbering{arabic}

%%%%%%%%%%%%%%%%%%%%%%%%%%%%%%
\section{Introduction}%%%%%%%%
%%%%%%%%%%%%%%%%%%%%%%%%%%%%%%
Multiple criteria decision aiding methods face three different types of problems: ranking, sorting and choice. A ranking problem consists in rank ordering of all considered alternatives, from the worst to the best, taking into account their evaluations on the considered criteria; a sorting problem consists in assigning each alternative to one or several predefined and ordered classes; a choice problem consists in selecting a set of alternatives considered the best or discarding a set of bad alternatives from the whole set of alternatives (for a detailed survey see \cite{FigGreEhr}). In order to deal with these problems, two different methodologies can be used: 
\begin{itemize}
\item assigning to each alternative a value given by a utility function, i.e. a real number reflecting its degree of desirability,
\item comparing the alternatives with respect to their evaluations on the considered criteria. 
\end{itemize}
In the first case, the Multi Attribute Utility Theory (MAUT) \cite{Keeney76} is most frequently used; MAUT provides a methodology for building a value function assigning a real number to each alternative. In the second case, the most popular are outranking methods which build an outranking relation in the set of alternatives that compare them pairwise \cite{Roy96_ENG,Brans84}. In the context of MAUT, one often uses additive value functions, that is functions obtained by adding up as many marginal value functions as there are evaluation criteria. In order to use this kind of procedure, one needs to construct all the marginal value functions. The construction requires some preference information elicited by the Decision Maker (DM). An analyst can obtain them in one of two ways: asking the DM to provide it directly, or indirectly. As direct definition of marginal value functions requires too big cognitive effort from the part of the DM, indirect elicitation of preference information prevails in the literature (see, e.g.,\cite{jacquet1982assessing},\cite{srinivasan1973estimating},\cite{pekelman1974mathematical}). Indirect preference information is expressed by the DM in terms of decision examples, e.g., holistic pairwise comparisons of some reference alternatives. When looking for an additive value function which is compatible with decision examples provided by the DM, i.e. reproduces decisions made by the DM, one can find many such compatible value functions, and, in general, each of these functions can give a different answer to the decision problem at hand. For this reason, Robust Ordinal Regression (ROR) was proposed (see, e.g.,\cite{greco2008ordinal},\cite{figueira2009building},\cite{greco2010multiple},\cite{Greco10}), that takes into consideration all compatible value functions simultaneously. In the context of ROR, two preference relations are considered: 
\begin{itemize}
\item possible preference relation, for which alternative $a$ is possibly preferred to alternative $b$ if $a$ is at least as good as $b$ for at least one compatible value function, and 
\item necessary preference relation, for which alternative $a$ is necessarily preferred to alternative $b$ if $a$ is at least as good as $b$ for all compatible value functions. 
\end{itemize}
Observe that ROR has been also used in case of preference models expressed in terms of outranking methods \cite{GKMS2010},\cite{Promethee2010}, in case of non-additive integrals \cite{angilella2010non}, so as also in presence of criteria structured in a hierarchical way \cite{CGSHierarchy}. 

In this paper, we introduce one important issue more to ROR: imprecise evaluations of alternatives. In many real world problems, alternatives are imprecisely evaluated with respect to the considered criteria; this is due to several reasons: for example, inexact definition of criteria, uncertainty or imprecision of data used for calculation of performances of alternatives on particular criteria, or subjective assessment of the performances. Different types and natures of imprecise evaluations have been taken into account in the literature; in the following we provide few examples: \cite{park1997tools} considers a situation in which imprecise evaluations concern weights, utilities, and in some cases also probabilities of the performances of alternatives on particular criteria and the DM provides this information in form of inequalities; \cite{kim1997group} considers not only imprecise weights and utilities, but dealing with a group decision problem, it takes into account imprecise importance of the members of the group; \cite{ahn2000multi},\cite{lee2002dominance} take into account imprecision of weights and utilities in a hierarchic context; \cite{eum2001establishing} takes into account imprecision regarding weights and utilities; \cite{fishburn1965analysis} compares decision strategies taking into account different types of imprecise information on the probabilities of states of the world; \cite{weber1987decision} explains different ways of facing situations with imprecise information on probabilities, utilities and evaluations with respect to considered criteria.

In this paper, we assume that imprecision regards the evaluations of alternatives with respect to the considered criteria; this implies that performance of an alternative on a considered criterion is not a unique number but an interval of possible values. First, we suppose that each interval is characterized by its extreme values, that is, the worst and the best evaluations the alternative can assume with respect to the considered criterion. For example, we could say that evaluation of alternative $a$ with respect to criterion $g$ is represented by the interval $[g^{L}(a),g^{R}(a)]$, to indicate that $a$ could have whatever evaluation between $g^{L}(a)$ and $g^{R}(a)$. Then, we will consider the case where each interval is characterized not only by its extreme values but also by another point between them; in this way $g(a)=[g^{L}(a),g^{P}(a),g^{R}(a)]$, and this interval indicates that, with respect to criterion $g$, alternative $a$ can assume whatever evaluation between $g^{L}(a)$ and $g^{R}(a)$ but the most probable evaluation assumed by $a$ is $g^{P}(a).$ Generalizing again, in a way of \cite{ozturk2010representing}, we will consider intervals characterized by their extreme values and by a certain number of other points between them, such that for a given criterion $g$, we have $g(a)=\left[g^{1}(a),\ldots,g^{n}(a)\right]$, where $g^{i}, i=1,\ldots,n$, will be called \textit{indicators} of criterion $g$. This means that an alternative could have any evaluation from among the $n$ listed points of the interval.

Obviously, if alternative $a$ assumes the same value for all indicators of criterion $g_{j}$, that is $g_{j}^{1}(a)=\ldots=g_{j}^{n}(a)$, then alternative $a$ has a precise evaluation on criterion $g_{j}$ being $g_{j}^{1}(a)$. 

Proceeding in this way, if $m$ is the number of evaluation criteria, to each alternative $a$ corresponds a vector of $m\times n$ values, and thus, instead of aggregating $m$ evaluations corresponding to the $m$ considered criteria, we have to aggregate $m\times n$ evaluations specified by the above mentioned indicators.

In order to deal with $m\times n$ possible evaluations of alternative $a$, we are considering $n$ fictitious copies of this alternative, denoted by $a^{(i)}, i=1,\ldots,n,$ where $a^{(i)}$ is an alternative having precise evaluations on all considered criteria and, in particular, for each criterion $g_{j}$ this precise evaluation is $g_{j}^{i}(a)$. For example, if alternative $a$ had the following evaluations on criteria $g_{1},g_{2}$ and $g_{3}$:

$$
g_{1}(a)=\left[10,15,17\right], \;\;\;\;\;\;\; g_{2}(a)=\left[23,35,50\right], \;\;\;\;\;\;\;\mbox{and} \;\;\;\;\;\;\; g_{3}(a)=\left[1,40,89\right], 
$$
then, fictitious alternatives $a^{(1)},a^{(2)}$ and $a^{(3)}$ were evaluated as shown in Table \ref{fict_alter}:

\begin{table}[!h]
\caption{Evaluation table of fictitious alternatives $a^{(1)},$ $a^{(2)}$, $a^{(3)}$}
\label{fict_alter}
\begin{center}
\begin{tabular}{|cccc|}
\hline
\rule[-3mm]{-1mm}{0.8cm}
alternative $a^{(i)}$ & $g_{1}\left(a^{(i)}\right)$ & $g_{2}\left(a^{(i)}\right)$ & $g_{3}\left(a^{(i)}\right)$ \\
\hline
\rule[-3mm]{-1mm}{0.8cm}
$a^{(1)}$ & [10,10,10] & [23,23,23] & [1,1,1] \\
\rule[-3mm]{-1mm}{0.8cm}
$a^{(2)}$ & [15,15,15] & [35,35,35] & [40,40,40] \\
\rule[-3mm]{-1mm}{0.8cm}
$a^{(3)}$ & [17,17,17] & [50,50,50] & [89,89,89] \\
\hline
\end{tabular}
\end{center}
\end{table}

\noindent Obviously, $a^{(1)}$ will be the worst realization of alternative $a$ because it has the worst possible evaluations on all considered criteria, and $a^{(n)}$ is instead the best realization of alternative $a$ because it has the best possible evaluations on all considered criteria. 

In the ROR context, we shall consider all the value functions compatible with the preference information provided by the DM, obtaining $n^{2}+1$ necessary and $n^{2}+1$ possible preference relations defined as follows: \\
For any pair $a,b$ of considered alternatives, 
\begin{itemize}
\item the $(i,k)$ necessary preference relation, with $i,k=1,\ldots,n$, for which $a$ is  $(i,k)-$necessarily preferred to $b$ if $a^{(i)}$ is at least as good as $b^{(k)}$ for all compatible value functions,
\item the $(i,k)$ possible preference relation, with $i,k=1,\ldots,n$, for which $a$ is  $(i,k)-$possibly preferred to $b$ if $a^{(i)}$ is at least as good as $b^{(k)}$ for at least one compatible value function,
\item the necessary preference relation, for which $a$ is necessarily preferred to $b$ if $a$ is at least as good as $b$ for all compatible value functions considering all $n$ evaluations of $a$ and all $n$ evaluations of $b$,
\item the possible preference relation, for which $a$ is possibly preferred to $b$ if $a$ is at least as good as $b$ for at least one compatible value function considering all $n$ evaluations of $a$ and all $n$ evaluations of $b$.
\end{itemize}
The paper is structured in the following way: section 2 describes basic concepts of imprecise evaluations; in section 3 we present the application of ROR in case of imprecise evaluations; section 4 contains the properties of necessary and possible preference relations, also in case of group decisions; in section 5 are present a few extensions of imprecise evaluations; section 6 contains a didactic example, and conclusions gathered in section 7 end the paper.

\section{Imprecise Evaluations - description of the model}

We are considering a decision problem in which a finite set of alternatives, denoted by $A=\left\{a,b,c,\ldots\right\}$, can have imprecise evaluations with respect to $m$ evaluation criteria $\left\{g_{1},\ldots,g_m\right\}$. In case of precise evaluations, a criterion $g_{j}$, $j\in J=\left\{1,\ldots,m\right\}$, is a function $g_{j}:A\rightarrow X_{j}$ where $X_{j}$ is the set of all possible evaluations (quantitative or qualitative, depending on the evaluation scale of criterion $g_{j}$) that an alternative could have on criterion $g_{j}$. In case of $n$-point imprecise evaluations we assume that each criterion is a function $g_{j}:A\rightarrow{\cal I}_{j}$, where ${\cal I}_{j}=\left\{(x_{1},\ldots,x_{n}):x_{l}\leq_{j}x_{l+1},l=1,\ldots,n-1\right\}\subseteq {X_{j}}^{n}$, and $\leq_{j}$ coincides with mathematical operator $\leq$ if criterion $f_j$ has a quantitative scale, however, it is defined differently if criterion $f_{j}$ has a qualitative scale. To explain this different definition of $\leq_{j}$, let us consider the evaluation of a student regarding a certain subject denoted by $g_{j}$, and suppose that the evaluations (s)he can have with respect to this subject are: ``very bad'', ``bad'', ``medium'', ``good'' and ``very good''. Then, we need to define an ordering $\leq_{j}$ between any two of these evaluations that is obviously different from the inequality $\leq$ between two real numbers, that is ``very bad'' $\leq_{j}$ ``bad'' $\leq_{j}$ ``medium'' $\leq_{j}$ ``good'' $\leq_{j}$ ``very good''. Using this notation, for each $a\in A$ and for each criterion $g_{j}, j\in J$, we say that $g_{j}(a)=\left[g_{j}^{1}(a),\ldots,g_{j}^{n}(a)\right]$, to indicate that alternative $a\in A$ could assume on criterion $g_{j}$ any evaluation from among $g_{j}^{1}(a),g_{j}^{2}(a),\ldots,g_{j}^{n}(a).$ Each $g_{j}^{i}$, $j=1,\ldots,m$, $i=1,\ldots,n$, will be called indicator. In this context it is worth noting one important remark:
\begin{itemize}
\item if alternative $a\in A$ has a precise evaluation on criterion $g_j$, $j\in J$, then $g_{j}^{1}(a)=\ldots=g_{j}^{n}(a).$
%\item if alternative $a\in A$ has a missing evaluation on criterion $g_{j},$ $j\in J$, that is, if we do not know the evaluation of alternative $a$ on criterion $g_{j}$ because it is impossible to define or calculate it, then we take into account in the interval of possible evaluations of $a$ only the worst and the best evaluations on this criterion, and we call these values $\alpha_j$ and $\beta_j$ respectively. %In this way, we do not express the interval of evaluations of $a$ with respect to criterion $g_{j}$, that is we do not give the values $g_{j}^{1}(a),\ldots,g_{j}^{n}(a)$, but we will consider this kind of missing data in a particular way when we will compute the preference relations. 
\end{itemize}

\noindent Without loss of generality, we will suppose the following: 
\begin{itemize}
\item each criterion has a quantitative scale, therefore $X_{j}\subseteq\rea$, and $\geq_{j}$ coincides with $\geq$,
\item for each $j=1,\ldots,m$, the interval $g_{j}(a)$ representing the evaluation of $a\in A$ on criterion $g_{j}$ is characterized by $n$ points,
\item the greater $g_{j}^{i}(a), a\in A$, the better is alternative $a$ on indicator $g_{j}^{i}, j=1,\ldots,m$, $i=1,\ldots,n.$
\end{itemize}

Considering the imprecise evaluations and the $n$ indicators $g_{j}^{i}$ for each criterion $g_{j}$, each alternative $a$ will be represented in the following way:
$$ g(a)=\left(\left[g_{1}^{1}(a),\ldots,g_{1}^{n}(a)\right],\ldots,\left[g_{j}^{1}(a),\ldots,g_{j}^{n}(a)\right],\ldots,\left[g_{m}^{1}(a),\ldots,g_{m}^{n}(a)\right]\right).$$

Using this notation, we can write the following definitions:

\begin{defn}\label{dominance}
Given alternatives $a,b\in A$, and $i,k\in\left\{1,\ldots,n\right\}$, we say that ``$a$ $(i,k)$-dominates $b$'', denoted by $a\Delta^{(i,k)}b$, if $g_{j}^{i}(a)\geq g_{j}^{k}(a),$ $\forall j=1,\ldots,m$. 
\end{defn}

\begin{defn}\label{Normal_dominance}
Given alternatives $a,b\in A,$ we say that ``$a$ normally dominates $b$'', denoted by $a\Delta b$, if $g_{j}^{i}(a)\geq g_{j}^{i}(b)$, $\forall j=1,\ldots,m$, and $\forall i=1,\ldots,n.$

\vspace{0,2cm}
\noindent Equivalently we can say that ``$a$ normally dominates $b$'' if $a$ $(i,i)$-dominates $b$, $\forall i=1,\ldots,n.$ 
\end{defn}

\noindent The following proposition gives some basic properties of dominance relations.

\begin{prop}\label{Dominance Proposition}
\hspace{0,1cm}
\begin{enumerate}
\item If $i\geq k,$ $i,k\in\left\{1,\ldots,n\right\}$, then $\Delta^{(i,k)}$ is reflexive, \label{dom_reflexivity}
\item If $i\leq k,$ $i,k\in\left\{1,\ldots,n\right\}$, then $\Delta^{(i,k)}$ is transitive, \label{dom_transitivity}
\item For each $i\in\left\{1,\ldots,n\right\},$ $\Delta^{(i,i)}$ is a partial preorder, \label{part_preorder}
\item If $r\geq i$ and $s\leq k,$ $i,k,r,s\in\left\{1,\ldots,n\right\}$, then $\Delta^{(i,k)}\;\subseteq\;\Delta^{(r,s)}$, \label{Dom_incl}
\item Given alternatives $a,b,c\in A$, if $a\Delta^{(i,k)}b,$ $b\Delta^{(i_1,k_1)}c,$ and $k\geq i_1$, $i,k,i_1,k_1\in\left\{1,\ldots,n\right\}$, then $a\Delta^{(r,s)}c$, $r,s\in\left\{1,\ldots,n\right\}:$ $r\geq i$ and $s\leq k_1,$ \label{dom_mixed_transitivity}
\item $\Delta$ is a partial preorder, \label{dom_preord}
\item Given alternatives $a,b,c\in A$, if $a\Delta^{(i,k)}b,$ $b\Delta c,$ $i,k\in\left\{1,\ldots,n\right\}$, then $a\Delta^{(s,t)}c$ with $s,t\in\left\{1,\ldots,n\right\}$ such that $s\geq i$ and $t\leq k$, \label{domind_dom_transitivity}
\item Given alternatives $a,b,c\in A$, if $a\Delta b,$ $b\Delta^{(i,k)} c,$ $i,k\in\left\{1,\ldots,n\right\}$, then $a\Delta^{(s,t)}c$ with $s,t\in\left\{1,\ldots,n\right\}$ such that $s\geq i$ and $t\leq k$. \label{dom_domind_transitivity}
\end{enumerate}
\end{prop}
\begin{proof}
See Appendix.
\end{proof}

\begin{note}
In the following, we shall call strong dominance, and we shall denote it by $\Delta^{S}$, the dominance relation $\Delta^{(1,n)}$. Similarly, we shall call weak dominance, and we shall denote it by $\Delta^{W}$, the dominance relation $\Delta^{(n,1)}$. When $n=2$, that is, when each interval is characterized by its worst and its best values only, then we only have the strong dominance relation, the normal dominance relation and the weak dominance relation. In this case, we can compare two alternatives considering only their best values and their worst values, or intervals of values provided by them. Using Proposition \ref{Dominance Proposition} we can state that weak and normal dominances are reflexive relations, normal and strong dominances are transitive relations, and so on.
\end{note}

Considering simultaneously the strong and the weak dominance relations we can state the following proposition:
\begin{prop}\label{salvo_prop}
\hspace{0,2cm}
\begin{enumerate}
\item $\Delta^{(1,n)} \;\subseteq\; \Delta \;\subseteq \; \Delta^{(n,1)},$ \label{dom_chain}
\item For $i,k=1,\ldots,n$, $\Delta^{(1,n)} \;\subseteq\; \Delta^{(i,k)} \;\subseteq \; \Delta^{(n,1)}.$ \label{dom_ind_chain}
\end{enumerate}
\end{prop}
\begin{proof}
See Appendix.
\end{proof}

\begin{note}
\noindent Proposition \ref{salvo_prop} shows how important it is to take into account the weak and the strong dominance relations, as they are the only two relations that can be compared directly with $\Delta$ because, in general, for any $(i,k)\in\left\{1,\ldots,n\right\}\times\left\{1,\ldots,n\right\}$, with $(i,k)\neq(1,n)$ and $(i,k)\neq(n,1)$, we can have $\Delta^{(i,k)}\not\subseteq\Delta$ and $\Delta\not\subseteq\Delta^{(i,k)}$. 
\end{note}

%%%%%%%%%%%%%%%%%%%%%%%%%%%%%%%%%%%%%%%%%%%%%%%%%%%%%%%%%%%%%%%%%%%%%%%%%%
\section{Robust Ordinal Regression for Imprecise Evaluations}\label{ROR}%%
%%%%%%%%%%%%%%%%%%%%%%%%%%%%%%%%%%%%%%%%%%%%%%%%%%%%%%%%%%%%%%%%%%%%%%%%%%

\noindent Considering only the dominance relations, the information stemming for the formulation of a multiple criteria decision problem is very poor. For this reason, in order to deal with one of the three classic decision problems (choice, ranking and sorting), we are using the Multi-attribute Utility Theory (MAUT) \cite{Keeney76}. MAUT considers value functions
$$
U(g_1(a), \ldots, g_m(a)) \colon \rea^m \rightarrow \rea
$$
such that: $a \,\,\text{is at least as good as}\,\, b \ %\text{with respect to criterion}
%\,\, i \,
\Leftrightarrow \ U(g_1(a), \ldots, g_m(a)) \ge U(g_1(b), \ldots, g_m(b))$, taking into account the evaluations of alternatives with respect to the considered criteria. In case of imprecise evaluations, we consider for each criterion $g_{j}, j\in J,$ $n$ indicators $g_{j}^{i}:A\rightarrow X_j$, $i=1,\ldots,n$, assigning to each alternative $a\in A$ the $i$-th evaluation from interval $g_{j}(a)$. Using this notation we can distinguish different types of value functions:

\begin{itemize}
\item $i$-th \textit{sub-marginal value function} with respect to the $i$-th indicator of criterion $g_{j}$, $u_{j,i}(g_{j}^{i}(a)):X_{j}\rightarrow\rea$, for all $j\in J$, and $i=1,\ldots,n,$ 
\item \textit{marginal value function} with respect to criterion $g_{j}$, $U_j\left(\left[g_{j}^{1}(a),\ldots,g_{j}^{n}(a)\right]\right):{\cal I}_{j}\rightarrow\rea$ such that $$U_{j}\left(\left[g_{j}^{1}(a),\ldots,g_{j}^{n}(a)\right]\right)=u_{j,1}\left(g_{j}^{1}(a)\right)+\ldots+u_{j,n}\left(g_{j}^{n}(a)\right),$$
\noindent (we are considering the marginal utility of alternative $a$ with respect to criterion $g_{j}$ dependent on all the $n$ considered indicators $g_{j}^{i}$ because each one of them gives a contribution to the evaluation of $a$ on criterion $g_{j}$ and this is also the reason for which we consider for each indicator $g_{j}^{i}$ a different sub-marginal value function $u_{j,i}$),
\item \textit{total additive value function}
$$U\left(\left[g_{1}^{1}(a)\ldots,g_{1}^{n}(a)\right],\ldots,\left[g_{m}^{1}(a),\ldots,g_{m}^{n}(a)\right]\right):{\cal I}_{1}\times\cdots\times{\cal I}_{m}\rightarrow\rea
$$ such that
$$
U\left(\left[g_{1}^{1}(a)\ldots,g_{1}^{n}(a)\right],\ldots,\left[g_{m}^{1}(a),\ldots,g_{m}^{n}(a)\right]\right)=
$$
$$
=\sum_{j=1}^{m}U_{j}\left(\left[g_{j}^{1}(a),\ldots,g_{j}^{n}(a)\right]\right)=\sum_{j=1}^{m}\left[\sum_{i=1}^{n}u_{j,i}\left(g_{j}^{i}(a)\right)\right]
$$
\end{itemize}
\noindent In the following, for the sake of simplicity, for each $j\in J$ we write $U_{j}(a)$ instead of  $U_j\left(\left[g_{j}^{1}(a),\ldots,g_{j}^{n}(a)\right]\right)$, and $U(a)$ instead of $U\left(\left[g_{1}^{1}(a),\ldots,g_{1}^{n}(a)\right],\ldots,\left[g_{m}^{1}(a),\ldots,g_{m}^{n}(a)\right]\right)$. In order to assign to each alternative a real number representing its degree of desirability, we need to know the sub-marginal value functions $u_{j,i}(\cdot),$ for all $j\in J$ and for all $i\in\left\{1,\ldots,n\right\}$. They can be obtained in two different ways: asking the decision maker (DM) what is the analytical expression of functions $u_{j,i}$, or computing them from indirect preference information elicited by the DM on a set $A^{R}\subseteq A$ of alternatives called \textit{reference alternatives}. The reference alternatives will be marked with a dash, like $\overline{a}$. We propose to use the second method, and thus the DM is asked to provide the following preference information:

\begin{itemize}
\item partial preorder $\succsim$ on $A^R$, whose meaning is: for $\overline{a},\overline{b}\in A^R$

\[
 \overline{a} \succsim \overline{b} \; \Leftrightarrow  \;  `` \overline{a} \;  \mbox{is at least as good as} \; \overline{b} \;\mbox{''};
\]
Taking into account $\succsim$, we have that $\succsim^{-1}$ denote the inverse of $\succsim$, i.e. if $\overline{a}\succsim \overline{b}$ then $\overline{b}\succsim^{-1}\overline{a};$ $\sim$ (indifference) is the symmetric part of $\succsim$  given by $\succsim\cap\succsim^{-1}$, i.e. if $\overline{a}\sim \overline{b}$ then $\overline{a}\succsim \overline{b}$ and $\overline{a}\succsim^{-1}\overline{b}$; $\succ$ (preference) is the asymmetric part of $\succsim$ given by $\succsim\setminus\sim$, i.e. if $\overline{a}\succ \overline{b}$ then $\overline{a}\succsim \overline{b}$ and not $\overline{a}\sim \overline{b}$;

\item partial preorder $\succsim^{\ast}$ on $A^R \times A^R$, whose meaning is: for $\overline{a},\overline{b},\overline{c},\overline{d} \in A^R$,

\[
(\overline{a},\overline{b}) \succsim^{\ast} (\overline{c}, \overline{d})  \Leftrightarrow   \;  \mbox{``$\overline{a}$ is preferred to $\overline{b}$ at least as much as $\overline{c}$ is preferred to $\overline{d}$\;''};
\]
analogously to $\succsim$, $\succ^{\ast}$ and $\sim^{\ast}$ are the asymmetric and the symmetric part of $\succsim^{\ast}$;

\item partial preorder $\succsim_{j}$ on $A^{R}$, whose meaning is: for $\overline{a},\overline{b} \in A^{R}$,
$$ \overline{a}\succsim_{j}\overline{b} \Leftrightarrow \;\;\mbox{``$\overline{a}$ is at least as good as $\overline{b}$ on criterion $g_{j}$''};$$ 
analogously to $\succsim$, $\succ_j$ and $\sim_j$ are the asymmetric and the symmetric part of $\succsim_j$;

\item partial preorder $\succsim_{j}^{\ast}$ on $A^{R} \times A^{R}$, whose meaning is: for $\overline{a},\overline{b},\overline{c},\overline{d} \in A^{R},$
$$(\overline{a},\overline{b})\succsim_{j}^{\ast} (\overline{c},\overline{d}) \Leftrightarrow \;\;\mbox{``$\overline{a}$ is preferred to $\overline{b}$ at least as much as $\overline{c}$ is preferred to $\overline{d}$}$$
$$ \mbox{ on criterion $g_{j}$''};$$
analogously to $\succsim$, $\succ_{j}^{\ast}$ and $\sim_{j}^{\ast}$ are the asymmetric and the symmetric part of $\succsim_{j}^{\ast}$.
\end{itemize}

In order to take into account the imprecise nature of evaluations, we consider for each alternative $a\in A$, $n$ fictitious alternatives $a^{(i)}$, having precise evaluations on all criteria, equal to the $i$-th point of interval $g_{j}(a)$, for each $j\in J$, i.e. $g_{j}^1\left(a^{(i)}\right)=\ldots=g_{j}^n\left(a^{(i)}\right)=g_{j}^i(a)$, for each $j\in J$. Note that given $a\in A$, a value function $U$ assigns to corresponding alternatives $a^{(i)}$ the value:

\begin{equation}\label{i-th alternative}
U\left(a^{(i)}\right)=u_{1,1}\left(g_{1}^{i}(a)\right)+\ldots+u_{1,n}\left(g_{1}^{i}(a)\right)+\ldots+u_{m,1}\left(g_{m}^{i}(a)\right)+\ldots+u_{m,n}\left(g_{m}^{i}(a)\right).
\end{equation}

%\noindent In consequence we can define a new set of alternatives 
%$$
%\tilde{A}=\displaystyle\cup_{a\in A}\left\{a,a^{(i)},i=1,\ldots,n\right\}.
%$$

An additive value function is called \textit{compatible} if it is able to restore the preference information supplied by the DM. Formally, a general additive compatible value function is an additive value function satisfying the following set of constraints:

\begin{equation*}
  \left .
       \begin{array}{l}
         \left .
         \begin{array}{l}

             U(\overline{a}) > U(\overline{b}) \;\;\; \text{if} \;\; \overline{a} \succ \overline{b}\\
             U(\overline{a}) = U(\overline{b}) \;\;\; \text{if} \;\; \overline{a} \sim \overline{b}\\
             U(\overline{a}) - U(\overline{b}) > U(\overline{c}) - U(\overline{d}) \;\;\; \text{if} \;\; (\overline{a},\overline{b}) \succ^{\ast} (\overline{c}, \overline{d}) \\
             U(\overline{a}) - U(\overline{b}) = U(\overline{c}) - U(\overline{d}) \;\;\; \text{if} \;\; (\overline{a},\overline{b}) \sim^{\ast} (\overline{c}, \overline{d}) \\ 
             U_{j}(\overline{a}) > U_{j}(\overline{b}) \;\;\; \text{if} \;\; \overline{a} \succ_{j} \overline{b}, \; j\in J \\
             U_{j}(\overline{a}) = U_{j}(\overline{b}) \;\;\; \text{if} \;\;  \overline{a} \sim_{j} \overline{b}, \; j\in J \\             
             U_{j}(\overline{a})-U_{j}(\overline{b})>U_{j}(\overline{c})-U_{j}(\overline{d}) \;\;\; \text{if} \;\; (\overline{a},\overline{b}) \succ^{\ast}_{j} (\overline{c},\overline{d}), \; j\in J\\
             U_{j}(\overline{a})-U_{j}(\overline{b})=U_{j}(\overline{c})-U_{j}(\overline{d}) \;\;\; \text{if} \;\; (\overline{a},\overline{b}) \sim^{\ast}_{j} (\overline{c},\overline{d}),  \; j\in J\\

         \end{array} \right \} \; \overline{a},\overline{b},\overline{c},\overline{d} \in A^{R}\\

          \;u_{j,i}(x_{j}^{k})-u_{j,i}(x_{j}^{k-1})\geq 0, \; \mbox{for each}\; j\in J,\;k=1, ..., m_{j}(A), \;i=1,\ldots,n\\
          \;u_{j,i}(x_{j}^{1})=0, \;\; \forall j\in J,i=1,\ldots,n\\
          \;\displaystyle\sum_{\substack{j\in J \\
i=1,\ldots,n}}u_{j,i}\left(x_{j}^{m_{j}(A)}\right)=1. \; \\
                \end{array}
       \right \} \; \left(E^{A^R}\right)
\end{equation*}

\noindent where, for each $j\in J$, $\displaystyle x_{j}^{1}=\min_{a\in A} g_{j}^{1}(a),$ $x_{j}^{m_{j}(A)}=\displaystyle\max_{a\in A} g_{j}^{n}(a)$, 
$x_{j}^{k}\in X_{j}({A}), k=1, ...,m_{j}(A),$ with $X_{j}(A)=\underset{\substack{a\in A \\ i=1,\ldots,n}}{\bigcup}g_{j}^{i}(a)\subseteq X_{j},$ being the set of all different evaluations of alternatives from $A$ on criterion $g_{j}, j\in J$, and $m_{j}(A)=\left|X_{j}(A)\right|.$ The values $x_{j}^{k}, k=1, ...,m_{j}(A),$ are increasingly ordered, i.e., $$x_{j}^{1}<x_{j}^{2}< ... <x_{j}^{m_{j}(A)-1}<x_{j}^{m_{j}(A)}.$$

In order to check the existence of a compatible value function, we have to weaken the strong inequalities by adding an auxiliary variable $\varepsilon,$ and then solve the following optimization problem, where the variables are $u_{j,i}(x_{j}^{k})$, $j\in J$, $i\in\left\{1,\ldots,n\right\}$, $k=1,\ldots,m_{j}(A)$ and $\varepsilon:$

\begin{equation*}
\begin{array}{l}
\;\;\;\mbox{Maximize} \;\;\varepsilon, \;\;\;\mbox{subject to constraints:}\\
\left.
\begin{array}{l}  %% Parentesi graffa grande
\left.
\begin{array}{l}
U(\overline{a}) \geq U(\overline{b})+\varepsilon \;\;\; \text{if} \;\; \overline{a} \succ \overline{b}\\
             U(\overline{a}) = U(\overline{b}) \;\;\; \text{if} \;\; \overline{a} \sim \overline{b}\\
             U(\overline{a}) - U(\overline{b}) \geq U(\overline{c}) - U(\overline{d}) +\varepsilon \;\;\; \text{if} \;\; (\overline{a},\overline{b}) \succ^{\ast} (\overline{c}, \overline{d}) \\
             U(\overline{a}) - U(\overline{b}) = U(\overline{c}) - U(\overline{d}) \;\;\; \text{if} \;\; (\overline{a},\overline{b}) \sim^{\ast} (\overline{c}, \overline{d}) \\ 
             U_{j}(\overline{a}) \geq U_{j}(\overline{b})+\varepsilon \;\;\; \text{if} \;\; \overline{a} \succ_{j} \overline{b}, \; j\in J \\
             U_{j}(\overline{a}) = U_{j}(\overline{b}) \;\;\; \text{if} \;\;  \overline{a} \sim_{j} \overline{b}, \; j\in J \\             
             U_{j}(\overline{a})-U_{j}(\overline{b})\geq U_{j}(\overline{c})-U_{j}(\overline{d})+\varepsilon \;\;\; \text{if} \;\; (\overline{a},\overline{b}) \succ^{\ast}_{j} (\overline{c},\overline{d}), \; j\in J\\
             U_{j}(\overline{a})-U_{j}(\overline{b})=U_{j}(\overline{c})-U_{j}(\overline{d}) \;\;\; \text{if} \;\; (\overline{a},\overline{b}) \sim^{\ast}_{j} (\overline{c},\overline{d}),  \; j\in J\\\end{array}
\right\}
\overline{a},\overline{b},\overline{c},\overline{d} \in A^{R}\\
\;\;u_{j,i}(x_{j}^{k})-u_{j,i}(x_{j}^{k-1})\geq 0, \;\mbox{for each}\; j\in J, \;k=1, ..., m_{j}(A), i=1,\ldots,n\\
\;\;u_{j,i}(x_{j}^{1})=0, \;\; \forall j\in J,i=1,\ldots,n\\
\;\;\displaystyle\sum_{\substack{j\in J \\
i=1,\ldots,n}}u_{j,i}\left(x_{j}^{m_{j}(A)}\right)=1. \; \\
\end{array}
\right\}\left(E^{A^{R'}}\right)
\end{array}
\end{equation*}

If $\varepsilon\left(E^{A^{R'}}\right)> 0$, where $\varepsilon\left(E^{A^{R'}}\right)=\max \varepsilon$,\; s.t. constraints $\left(E^{A^{R'}}\right)$,  then there exists at least one compatible value function $U(\cdot)$; if instead, $\varepsilon\left(E^{A^{R'}}\right) \leq 0$, then there does not exist any compatible value function $U(\cdot)$. Supposing that there exist more than one compatible value function, we indicate by ${\cal U}$ the set of all compatible value functions; in general, each of these functions will induce a different ranking on the set $A$ of alternatives. For this reason, Robust Ordinal Regression methods (see \cite{greco2008ordinal},\cite{figueira2009building},\cite{greco2010multiple},\cite{GKMS2010},\cite{Promethee2010},\cite{angilella2010non}), do not take into account only one compatible value function but the whole set of compatible value functions simultaneously. 

Taking into account the imprecise nature of evaluations, and considering for each alternative $a\in A$ the fictitious alternatives $a^{(i)}$, $i=1,\ldots,n$, we can define the following preference relations:

\begin{defn}\label{pos_pref}
Given two alternatives $a,b\in A$ and the set ${\cal U}$ of compatible value functions on $A^{R}\subseteq A,$ we say that $a$ is possibly preferred to $b$, if $a$ is at least as good as $b$ for at least one compatible value function:

$$a\succsim^{P}b \Leftrightarrow \;\mbox{there exists}\; U\in{\cal U}: U(a)\geq U(b).$$

\end{defn}

\begin{defn}\label{nec_pref}
Given two alternatives $a,b\in A$ and the set ${\cal U}$ of compatible value functions on $A^{R}\subseteq A,$ we say that $a$ is necessarily preferred to $b$, if $a$ is at least as good as $b$ for all compatible value functions:

$$ a\succsim^{N}b \Leftrightarrow U(a)\geq U(b), \;\mbox{for all}\; U\in{\cal U}.$$

\end{defn}

\begin{defn}
Given two alternatives $a,b\in A$, the set ${\cal U}$ of compatible value functions on $A^{R}\subseteq A,$ and $i,k\in\left\{1,\ldots,n\right\}$, we say that $a$ is $(i,k)$-possibly preferred to $b$, if $a^{(i)}$ is at least as good as $b^{(k)}$ for at least one compatible value function:

$$a\succsim_{(i,k)}^{P}b \Leftrightarrow \;\mbox{there exists}\; U\in{\cal U}: U\left(a^{(i)}\right)\geq U\left(b^{(k)}\right).$$

\end{defn}

\begin{defn}
Given two alternatives $a,b\in A$, the set ${\cal U}$ of compatible value functions on $A^{R}\subseteq A,$ and $i,k\in\left\{1,\ldots,n\right\}$, we say that $a$ is $(i,k)-$necessarily preferred to $b$, if $a^{(i)}$ is at least as good as $b^{(k)}$ for all compatible value functions:

$$ a\succsim_{(i,k)}^{N}b \Leftrightarrow U\left(a^{(i)}\right)\geq U\left(b^{(k)}\right), \;\mbox{for all}\; U\in{\cal U}.$$

\end{defn}

\noindent For all $a,b\in A$, and for all $i,k\in\left\{1,\ldots,n\right\}$, let us consider the following sets of constraints:

\begin{equation*}
\left.
\begin{array}{l}
U(a)\geq U(b)\\
E^{A^{R'}}
\end{array}
\right\}\left(E^{P}(a,b)\right), \;\;\;\;\;\;\;\;\;\;\;\;\;\;\;
\left.
\begin{array}{l}
U(b)\geq U(a)+\varepsilon\\
E^{A^{R'}}
\end{array}
\right\}\left(E^{N}(a,b)\right), 
\end{equation*}

\vspace{0,4cm}
\begin{equation*}
\left.
\begin{array}{l}
U(a^{(i)})\geq U(b^{(k)})\\
E^{A^{R'}}
\end{array}
\right\}\left(E_{(i,k)}^{P}(a,b)\right), \;\;\;\;\;\;\;\;\;\;\;\;\;\;\;
\left.
\begin{array}{l}
U(b^{(k)})\geq U(a^{(i)})+\varepsilon\\
E^{A^{R'}}
\end{array}
\right\}\left(E_{(i,k)}^{N}(a,b)\right).
\end{equation*}

\noindent then: 
\begin{itemize}

\item $a\succsim^{P}b$ if $E^{P}(a,b)$ is feasible and $\varepsilon^{P}(a,b)>0,$ where $\varepsilon^{P}(a,b)=\max\varepsilon$, \;\;s.t. constraints $E^{P}(a,b)$,

\item $a\succsim^{N}b$ if $E^{N}(a,b)$ is infeasible or $\varepsilon^{N}(a,b)\leq 0,$ where $\varepsilon^{N}(a,b)=\max \varepsilon$, \;\;s.t. constraints $E^{N}(a,b)$,

\item $a\succsim_{(i,k)}^{P}b$ if $E_{(i,k)}^{P}(a,b)$ is feasible and $\varepsilon_{(i,k)}^{P}(a,b)>0,$ where $\varepsilon_{(i,k)}^{P}(a,b)=\max \varepsilon$, \;\;s.t. constraints $E_{(i,k)}^{P}(a,b)$,

\item $a\succsim_{(i,k)}^{N}b$ if $E_{(i,k)}^{N}(a,b)$ is infeasible or $\varepsilon_{(i,k)}^{N}(a,b)\leq 0,$ where $\varepsilon_{(i,k)}^{N}(a,b)=\max\varepsilon$, \;\;s.t. constraints $E_{(i,k)}^{N}(a,b)$.

\end{itemize}

%%%%%%%%%%%%%%%%%%%%%%%%%%%%%%%%%%%%%%%%%%%%%%%%%%%%%%%%%%%%%%%%%%%%%%
\section{Properties of necessary and possible preference relations}%%%
%%%%%%%%%%%%%%%%%%%%%%%%%%%%%%%%%%%%%%%%%%%%%%%%%%%%%%%%%%%%%%%%%%%%%%

\noindent The necessary and possible preference relations satisfy several interesting properties presented in the following propositions.

\begin{prop}\label{Prop_relations}
\hspace{0,1cm}
\begin{enumerate}
\item For each $a\in A,$ if $i\geq k$, $i,k\in\left\{1,\ldots,n\right\}$, then $U(a^{(i)})\geq U(a^{(k)}),$ \label{monotonicity}
\item For each $a\in A,$ $U(a^{(1)})\leq U(a)\leq U(a^{(n)}),$ \label{chain}
\item $\succsim^{N}\;\subseteq\;\succsim^{P}$, \label{Normal_Inclusion}
\item For all $i,k\in\left\{1,\ldots,n\right\}$, $\succsim_{(i,k)}^{N}\;\subseteq\;\succsim_{(i,k)}^{P}$, \label{inclus_nec_pos_n_point}
\item $\succsim^{N}$ is a partial preorder, i.e. a reflexive and transitive binary relation, \label{Normal_relation}
\item $\succsim^{P}$ is strongly complete and negatively transitive, \label{Possible_relation}
\item For all $a,b\in A$, $a\succsim^{N}b$ or $b\succsim^{P}a$, \label{comp_nec_pos}
\item $\Delta\;\subseteq\;\succsim^{N}$, \label{inclusion_dominance_nec}
\item Given $a,b,c\in A$ such that $a\succsim^{N} b$ and $b\succsim^{P}c$, then $a\succsim^{P}c$, \label{necc_pos}
\item Given $a,b,c\in A$ such that $a\succsim^{P} b$ and $b\succsim^{N}c$, then $a\succsim^{P}c$, \label{pos_necc}
\item Given $a,b,c\in A$ such that $a\Delta b$ and $b\succsim^{N}c$, then $a\succsim^{N}c$, \label{delta_nec}
\item Given $a,b,c\in A$ such that $a\succsim^{N} b$ and $b\Delta c$, then $a\succsim^{N}c$, \label{nec_delta}
\item Given $a,b,c\in A$ such that $a\Delta b$ and $b\succsim^{P}c$, then $a\succsim^{P}c$, \label{delta_pos}
\item Given $a,b,c\in A$ such that $a\succsim^{P} b$ and $b\Delta c$, then $a\succsim^{P}c$, \label{pos_delta}
\item For all $i,k\in\left\{1,\ldots,n\right\}$, $\Delta^{(i,k)}\;\subseteq\;\succsim_{(i,k)}^{N},$ \label{inclusion_dominance}
\item Given $a,b,c\in A$, $i,r\in\left\{1,\ldots,n\right\}$ such that $a\succsim_{(i,n)}^{N} b$, $b\succsim^{N}c$ and $r\geq i$, then $a\succsim_{(r,1)}^{N}c$,\label{nec_ind_nec}
\item Given $a,b,c\in A$, $k,r\in\left\{1,\ldots,n\right\}$ such that $a\succsim^{N} b$, $b\succsim_{(1,k)}^{N}c$ and $r\leq k$, then $a\succsim_{(n,r)}^{N}c$, \label{nec_nec_ind}
\item Given $a,b,c\in A$, $i,r\in\left\{1,\ldots,n\right\}$ such that $a\succsim_{(i,n)}^{P} b$, $b\succsim^{N}c$ and $r\geq i$, then $a\succsim_{(r,1)}^{P}c$, \label{pos_ind_nec}
\item Given $a,b,c\in A$, $k,r\in\left\{1,\ldots,n\right\}$ such that $a\succsim^{N} b$, $b\succsim_{(1,k)}^{P}c$ and $r\leq k$, then $a\succsim_{(n,r)}^{P}c$, \label{nec_pos_ind}
\item Given $a,b,c\in A$, $i,r\in\left\{1,\ldots,n\right\}$ such that $a\succsim_{(i,n)}^{N} b$, $b\succsim^{P}c$ and $r\geq i$, then $a\succsim_{(r,1)}^{P}c$, \label{nec_ind_pos}
\item Given $a,b,c\in A$, $k,r\in\left\{1,\ldots,n\right\}$ such that $a\succsim^{P} b$, $b\succsim_{(1,k)}^{N}c$ and $r\leq k$, then $a\succsim_{(n,r)}^{P}c$, \label{pos_nec_ind}
\item Given $a,b,c\in A$, $i,r\in\left\{1,\ldots,n\right\}$ such that $a\Delta^{(i,n)}b$, $b\succsim^{N}c$ and $r\geq i$, then $a\succsim_{(r,1)}^{N}c$, \label{delta_ind_nec}
\item Given $a,b,c\in A$, $k,r\in\left\{1,\ldots,n\right\}$ such that $a\succsim^{N} b$, $b\Delta^{(1,k)}c$ and $r\leq k$, then $a\succsim_{(n,r)}^{N}c$, \label{nec_dom_ind}
\item Given $a,b,c\in A$, $i,r\in\left\{1,\ldots,n\right\}$ such that $a\Delta^{(i,n)}b$, $b\succsim^{P}c$ and $r\geq i$, then $a\succsim_{(r,1)}^{P}c$, \label{delta_ind_pos}
\item Given $a,b,c\in A$, $k,r\in\left\{1,\ldots,n\right\}$ such that $a\succsim^{P} b$, $b\Delta^{(1,k)}c$ and $r\leq k$, then $a\succsim_{(n,r)}^{P}c$, \label{pos_dom_ind}
\item Given $a,b,c\in A$, $k,r\in\left\{1,\ldots,n\right\}$ such that $a\Delta b$, $b\succsim_{(1,k)}^{N}c$ and $r\leq k$, then $a\succsim_{(n,r)}^{N}c$, \label{delta_nec_ind}
\item Given $a,b,c\in A$, $i,r\in\left\{1,\ldots,n\right\}$ such that $a\succsim_{(i,n)}^{N} b$, $b\Delta c$ and $r\geq i$, then $a\succsim_{(r,1)}^{N}c$, \label{nec_ind_dom}
\item Given $a,b,c\in A$, $k,r\in\left\{1,\ldots,n\right\}$ such that $a\Delta b$, $b\succsim_{(1,k)}^{P}c$ and $r\leq k$, then $a\succsim_{(n,r)}^{P}c$, \label{delta_pos_ind}
\item Given $a,b,c\in A$, $i,r\in\left\{1,\ldots,n\right\}$ such that $a\succsim_{(i,n)}^{P} b$, $b\Delta c$ and $r\geq i$, then $a\succsim_{(r,1)}^{P}c$, \label{pos_ind_dom}
\item If $i\geq k$, $i,k\in\left\{1,\ldots,n\right\}$, then $\succsim_{(i,k)}^{N}$ is reflexive, \label{reflexivity}
\item If $i\leq k$, $i,k\in\left\{1,\ldots,n\right\}$, then $\succsim_{(i,k)}^{N}$ is transitive, \label{transitivity}
\item For each $i\in\left\{1,\ldots,n\right\}$, $\succsim_{(i,i)}^{N}$ is a partial preorder, \label{impr_partial_preord}
\item For all $a,b\in A$, for all $i,k\in\left\{1,\ldots,n\right\}$, we have $a\succsim_{(i,k)}^{N}b$ or $b\succsim_{(k,i)}^{P}a$, \label{imprec_completeness}  
\item If $i\geq k$, $i,k\in\left\{1,\ldots,n\right\}$, then $\succsim_{(i,k)}^{P}$ is strongly complete and negatively transitive, \label{refl_negtran_imp_poss}
\item If $i_1\geq i$ and $k_1\leq k$, $i,k,i_1,k_1\in\left\{1,\dots,n\right\}$, then $\succsim_{(i,k)}^{N}\;\subseteq\;\succsim_{(i_1,k_1)}^{N}$, \label{nec_inclusion}
\item If $i_1\geq i$ and $k_1\leq k$, $i,k,i_1,k_1\in\left\{1,\dots,n\right\}$, then $\succsim_{(i,k)}^{P}\;\subseteq\;\succsim_{(i_1,k_1)}^{P}$, \label{pos_inclusion}
\item For all $i,k=1,\ldots,n$, $\succsim_{(1,n)}^{N}\;\subseteq\;\succsim_{(i,k)}^{N}\;\subseteq\;\succsim_{(n,1)}^{N},$ \label{N_incl_1}
\item For all $i,k=1,\ldots,n$, $\succsim_{(1,n)}^{P}\;\subseteq\;\succsim_{(i,k)}^{P}\;\subseteq\;\succsim_{(n,1)}^{P},$ \label{P_incl_1}
\item $\succsim_{(1,n)}^{N}\;\subseteq\;\succsim^{N}\;\subseteq\;\succsim_{(n,1)}^{N},$ \label{N_incl}
\item $\succsim_{(1,n)}^{P}\;\subseteq\;\succsim^{P}\;\subseteq\;\succsim_{(n,1)}^{P},$ \label{P_incl}
\item If $a\succsim_{(i,k)}^{N}b,$ $b\succsim_{(i_1,k_1)}^{N}c,$ and $k\geq i_1$, with $i,k,i_1,k_1\in\left\{1,\ldots,n\right\}$, then $a\succsim_{(r,s)}^{N}c$, where $r,s\in\left\{1,\ldots,n\right\}:$ $r\geq i$ and $s\leq k_1,$ \label{nec_transitivity}
\item If $a\succsim_{(i,k)}^{N}b,$ $b\succsim_{(i_1,k_1)}^{P}c,$ and $k\geq i_1$, with $i,k,i_1,k_1\in\left\{1,\ldots,n\right\}$, then $a\succsim_{(r,s)}^{P}c$, where $r,s\in\left\{1,\ldots,n\right\}:$ $r\geq i$ and $s\leq k_1,$ \label{necpos_transitivity}
\item If $a\succsim_{(i,k)}^{P}b,$ $b\succsim_{(i_1,k_1)}^{N}c,$ and $k\geq i_1$, with $i,k,i_1,k_1\in\left\{1,\ldots,n\right\}$, then $a\succsim_{(r,s)}^{P}c$, where $r,s\in\left\{1,\ldots,n\right\}:$ $r\geq i$ and $s\leq k_1,$ \label{posnec_transitivity}
\item If $a\Delta^{(i,k)}b,$ $b\succsim_{(i_1,k_1)}^{N}c,$ and $k\geq i_1$, with $i,k,i_1,k_1\in\left\{1,\ldots,n\right\}$, then $a\succsim_{(r,s)}^{N}c$, where $r,s\in\left\{1,\ldots,n\right\}:$ $r\geq i$ and $s\leq k_1,$ 
\item If $a\succsim_{(i,k)}^{N}b,$ $b\Delta^{(i_1,k_1)}c,$ and $k\geq i_1$, with $i,k,i_1,k_1\in\left\{1,\ldots,n\right\}$, then $a\succsim_{(r,s)}^{N}c$, 
where $r,s\in\left\{1,\ldots,n\right\}:$ $r\geq i$ and $s\leq k_1,$
\item If $a\Delta^{(i,k)}b,$ $b\succsim_{(i_1,k_1)}^{P}c,$ and $k\geq i_1$, with $i,k,i_1,k_1\in\left\{1,\ldots,n\right\}$, then $a\succsim_{(r,s)}^{P}c$, where $r,s\in\left\{1,\ldots,n\right\}:$ $r\geq i$ and $s\leq k_1,$
\item If $a\succsim_{(i,k)}^{P}b,$ $b\Delta^{(i_1,k_1)}c,$ and $k\geq i_1$, with $i,k,i_1,k_1\in\left\{1,\ldots,n\right\}$, then $a\succsim_{(r,s)}^{P}c$, where $r,s\in\left\{1,\ldots,n\right\}:$ $r\geq i$ and $s\leq k_1,$
\end{enumerate}
\end{prop}
\begin{proof}
See appendix.
\end{proof}

\begin{note}
Later, we will also use the following notation:
\begin{itemize}
\item ``strongly necessary preference relation'', to indicate the necessary preference relation $\succsim_{(1,n)}^{N}$ (denoted by $\succsim^{SN}$),
\item ``strongly possible preference relation'', to indicate the possible preference relation $\succsim_{(1,n)}^{P}$ (denoted by $\succsim^{SP}$),
\item ``weakly necessary preference relation'', to indicate the necessary preference relation $\succsim_{(n,1)}^{N}$ (denoted by $\succsim^{WN}$), and
\item ``weakly possible preference relation'', to indicate the possible preference relation $\succsim_{(n,1)}^{P}$ (denoted by $\succsim^{WP}$).
\end{itemize} 
In case of $n=2$-point intervals, the only preference relations we can consider are the four preference relations cited above, as well as the necessary and possible preference relations from Definitions (\ref{pos_pref}) and (\ref{nec_pref}).
\end{note}

Considering the weak, normal and strong dominance relations, together with weak, normal and strong preference relations as a straightforward consequence of Proposition \ref{Prop_relations} and Proposition \ref{salvo_prop} we obtain the set of dependencies shown in Figure \ref{Consequences}:
\begin{prop}\label{picture_prop}
Implications in Figure \ref{Consequences} hold.

\begin{figure}[!ht]
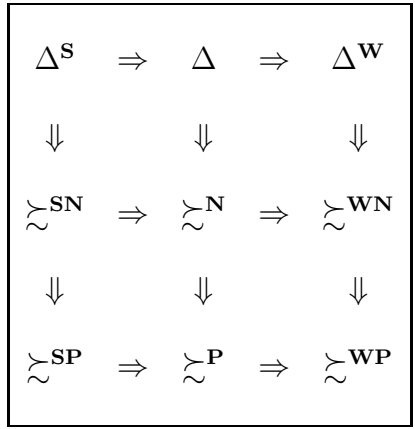

\caption{Dependencies between all kinds of dominance relations and preference relations}\label{Consequences}
\begin{center}
\begin{tabular}{|ccccc|}
\hline
&&&&\\
$\Delta^{\mathbf{S}}$ & $\mathbf{\Rightarrow}$ & $\Delta$ & $\mathbf{\Rightarrow}$ & $\Delta^{\mathbf{W}}$\\
&&&&\\
$\mathbf{\Downarrow}$ & & $\mathbf{\Downarrow}$ & & $\mathbf{\Downarrow}$ \\
 & & & & \\
$\mathbf{\succsim^{SN}}$ & $\mathbf{\Rightarrow}$ & $\mathbf{\succsim^{N}}$ & $\mathbf{\Rightarrow}$ & $\mathbf{\succsim^{WN}}$\\
 & & & & \\
$\mathbf{\Downarrow}$ & & $\mathbf{\Downarrow}$ & & $\mathbf{\Downarrow}$ \\
 & & & & \\
$\mathbf{\succsim^{SP}}$ & $\mathbf{\Rightarrow}$ & $\mathbf{\succsim^{P}}$ & $\mathbf{\Rightarrow}$ & $\mathbf{\succsim^{WP}}$\\
&&&&\\
\hline
\end{tabular}
\end{center}
\end{figure}
Moreover:
\begin{enumerate}
\item $a\succsim^{SN}b$ and $b\succsim^{N}c$ $\Rightarrow$ $a\succsim^{N}c$, \label{SN_N}
\item $a\succsim^{SN}b$ and $b\succsim^{P}c$ $\Rightarrow$ $a\succsim^{P}c$, \label{SN_P}
\item $a\succsim^{N}b$ and $b\succsim^{SN}c$ $\Rightarrow$ $a\succsim^{N}c$, \label{N_SN}
\item $a\succsim^{N}b$ and $b\succsim^{SP}c$ $\Rightarrow$ $a\succsim^{P}c$, \label{N_SP}
\item $a\succsim^{SP}b$ and $b\succsim^{N}c$ $\Rightarrow$ $a\succsim^{P}c$, \label{SP_N}
\item $a\succsim^{P}b$ and $b\succsim^{SN}c$ $\Rightarrow$ $a\succsim^{P}c$, \label{P_SN}
\item $a\succsim^{N}b$ or $b\succsim^{P}a,$
\item $a\succsim^{SN}b$ or $b\succsim^{WP}a,$
\item $a\succsim^{WN}b$ or $b\succsim^{SP}a,$
\item $a\Delta^{S}b$ and $b\Delta c$ $\Rightarrow$ $a\Delta^{S}c$,
\item $a\Delta^{S}b$ and $b\Delta^{W}c$ $\Rightarrow$ $a\Delta^{W}c$,
\item $a\Delta^{S}b$ and $b\succsim^{SN}c$ $\Rightarrow$ $a\succsim^{SN}c$,
\item $a\Delta^{S}b$ and $b\succsim^{N}c$ $\Rightarrow$ $a\succsim^{N}c$,
\item $a\Delta^{S}b$ and $b\succsim^{WN}c$ $\Rightarrow$ $a\succsim^{WN}c$,
\item $a\Delta^{S}b$ and $b\succsim^{SP}c$ $\Rightarrow$ $a\succsim^{SP}c$,
\item $a\Delta^{S}b$ and $b\succsim^{P}c$ $\Rightarrow$ $a\succsim^{P}c$,
\item $a\Delta^{S}b$ and $b\succsim^{WP}c$ $\Rightarrow$ $a\succsim^{WP}c$,
\item $a\Delta b$ and $b\Delta^{S}c$ $\Rightarrow$ $a\Delta^{S}c$,
\item $a\Delta b$ and $b\Delta^{W}c$ $\Rightarrow$ $a\Delta^{W}c$,
\item $a\Delta b$ and $b\succsim^{SN}c$ $\Rightarrow$ $a\succsim^{N}c$,
\item $a\Delta b$ and $b\succsim^{SP}c$ $\Rightarrow$ $a\succsim^{P}c$,
\item $a\Delta^{W}b$ and $b\Delta c$ $\Rightarrow$ $a\Delta^{W}c$,
\item $a\Delta^{W}b$ and $b\succsim^{SN}c$ $\Rightarrow$ $a\Delta^{WN}c$,
\item $a\Delta^{W}b$ and $b\succsim^{SP}c$ $\Rightarrow$ $a\succsim^{WP}c$,
\item $a\succsim^{SN}b$ and $b\Delta^{S}c$ $\Rightarrow$ $a\succsim^{SN}c$,
\item $a\succsim^{SN}b$ and $b\Delta c$ $\Rightarrow$ $a\succsim^{N}c$,
\item $a\succsim^{SN}b$ and $b\Delta^{W}c$ $\Rightarrow$ $a\succsim^{WN}c$,
\item $a\succsim^{SN}b$ and $b\succsim^{WN}c$ $\Rightarrow$ $a\succsim^{WN}c$,
\item $a\succsim^{SN}b$ and $b\succsim^{SP}c$ $\Rightarrow$ $a\succsim^{SP}c$,
\item $a\succsim^{SN}b$ and $b\succsim^{WP}c$ $\Rightarrow$ $a\succsim^{WP}c$,
\item $a\succsim^{N}b$ and $b\Delta^{S}c$ $\Rightarrow$ $a\succsim^{N}c$,
\item $a\succsim^{WN}b$ and $b\Delta^{S}c$ $\Rightarrow$ $a\succsim^{WN}c$,
\item $a\succsim^{WN}b$ and $b\succsim^{SN}c$ $\Rightarrow$ $a\succsim^{WN}c$,
\item $a\succsim^{WN}b$ and $b\succsim^{SP}c$ $\Rightarrow$ $a\succsim^{WP}c$,
\item $a\succsim^{SP}b$ and $b\Delta^{S}c$ $\Rightarrow$ $a\succsim^{SP}c$,
\item $a\succsim^{SP}b$ and $b\Delta c$ $\Rightarrow$ $a\succsim^{P}c$,
\item $a\succsim^{SP}b$ and $b\Delta^{W}c$ $\Rightarrow$ $a\succsim^{WP}c$,
\item $a\succsim^{SP}b$ and $b\succsim^{SN}c$ $\Rightarrow$ $a\succsim^{SP}c$,
\item $a\succsim^{SP}b$ and $b\succsim^{WN}c$ $\Rightarrow$ $a\succsim^{WP}c$,
\item $a\succsim^{P}b$ and $b\Delta^{S}c$ $\Rightarrow$ $a\succsim^{P}c$,
\item $a\succsim^{WP}b$ and $b\Delta^{S}c$ $\Rightarrow$ $a\succsim^{WP}c$,
\item $a\succsim^{WP}b$ and $b\succsim^{SN}c$ $\Rightarrow$ $a\succsim^{WP}c$.
\end{enumerate}
\end{prop}
\begin{proof}
See Appendix.
\end{proof}

%%%%%%%%%%%%%%%%%%%%%%%%%%%%%%
\vspace{0,5cm}%%%%%%%%%%%%%%%%
\subsection{Group Decisions}%%
%%%%%%%%%%%%%%%%%%%%%%%%%%%%%%
In many decision making problems, there are more than one DM. For example, in case of decisions related to land development, a group of stakeholders  with different perceptions of predefined criteria has to be involved. Robust Ordinal Regression (\cite{greco2009possible,ROR_group}) has been applied to group decisions as follows. Considering a set ${\cal DM}$ of DMs, and a set of pairwise comparisons provided by the DMs belonging to ${\cal D^{'}}\subseteq{\cal DM}$, for each $d_{h}\in {\cal D^{'}}$ we compute the necessary and possible preference relations $\succsim^{N}_{h}$ and $\succsim^{P}_{h}$. Then, we can represent consensus between decision makers from ${\cal DM}$, defining the following preference relations for all ${\cal D'}\subseteq{\cal DM}$:
\begin{itemize}
\item the necessary-necessary preference relation ($\succsim^{N,N}_{\cal D^{'}}$), for which $a$ is necessarily preferred to $b$ for all $d_{h}\in{\cal D^{'}}$,
\item the necessary-possible preference relation ($\succsim^{N,P}_{\cal D^{'}}$), for which $a$ is necessarily preferred to $b$ for at least one $d_{h}\in{\cal D^{'}}$,
\item the possible-necessary preference relation ($\succsim^{P,N}_{\cal D^{'}}$), for which $a$ is possibly preferred to $b$ for all $d_{h}\in{\cal D^{'}}$,
\item the possible-possible preference relation ($\succsim^{P,P}_{\cal D^{'}}$), for which $a$ is possibly preferred to $b$ for at least one $d_{h}\in{\cal D^{'}}$.
\end{itemize}

\noindent The above preference relations $\succsim^{N,N}_{\cal D^{'}}$, $\succsim^{N,P}_{\cal D^{'}}$, $\succsim^{P,N}_{\cal D^{'}}$ and $\succsim^{P,P}_{\cal D^{'}}$, extensively discussed in \cite{ROR_group}, satisfy some interesting properties, as for example:
\begin{itemize}
\item $\succsim^{N,N}_{\cal D^{'}} \;\subseteq \; \succsim^{N,P}_{\cal D^{'}} \; \subseteq \; \succsim^{P,P}_{\cal D^{'}}$, 
\item $\succsim^{N,N}_{\cal D^{'}} \;\subseteq \; \succsim^{P,N}_{\cal D^{'}} \; \subseteq \; \succsim^{P,P}_{\cal D^{'}}$
\end{itemize}

In case of $n$-point imprecise evaluations, we can extend the number of preference relations considered in group decisions, giving the following definition:

\begin{defn}
\noindent For all $i,k\in\left\{1,\ldots,n\right\}$, and for all $a,b\in A:$
\begin{itemize}
\item $a\succsim^{N,N}_{(i,k),{\cal D^{'}}}b$ if $a$ is $(i,k)$-necessarily preferred to $b$ for all DMs $d_{h}\in{\cal D^{'}},$
\item $a\succsim^{N,P}_{(i,k),{\cal D^{'}}}b$ if $a$ is $(i,k)$-necessarily preferred to $b$ for at least one DM $d_{h}\in{\cal D^{'}},$
\item $a\succsim^{P,N}_{(i,k),{\cal D^{'}}}b$ if $a$ is $(i,k)$-possibly preferred to $b$ for all DMs $d_{h}\in{\cal D^{'}},$
\item $a\succsim^{P,P}_{(i,k),{\cal D^{'}}}b$ if $a$ is $(i,k)$-possibly preferred to $b$ for at least one DM $d_{h}\in{\cal D^{'}}.$
\end{itemize}
\end{defn}

\begin{note}
Analogously to the case of a single DM, we can state the following equalities:
\begin{itemize}
\item $\succsim_{(1,n),{\cal D^{'}}}^{N,N}=\succsim^{SN,N}_{\cal D^{'}}$, \;\; $\succsim_{(1,n),{\cal D^{'}}}^{N,P}=\succsim^{SN,P}_{\cal D^{'}}$,\;\; $\succsim_{(1,n),{\cal D^{'}}}^{P,N}=\succsim^{SP,N}_{\cal D^{'}}$, \;\; $\succsim_{(1,n),{\cal D^{'}}}^{P,P}=\succsim^{SP,P}_{\cal D^{'}}$,
\item $\succsim_{(n,1),{\cal D^{'}}}^{N,N}=\succsim^{WN,N}_{\cal D^{'}}$,\;\; $\succsim_{(n,1),{\cal D^{'}}}^{N,P}=\succsim^{WN,N}_{\cal D^{'}}$, \;\; $\succsim_{(n,1),{\cal D^{'}}}^{P,N}=\succsim^{WP,N}_{\cal D^{'}}$, \;\; $\succsim_{(n,1),{\cal D^{'}}}^{P,P}=\succsim^{WP,P}_{\cal D^{'}}$.
\end{itemize}
\end{note}

In the case of $n$-point imprecise evaluations we can formulate the following proposition:

\begin{prop}\label{group_property}
Given $i,k,i_1,k_1\in\left\{1,\ldots,n\right\}$, ${\cal D^{'}}\subseteq{\cal D}$ and $R_{1},R_{2},R_{1}^{'},R_{2}^{'}\in\left\{ P,N\right\}$ such that:
\begin{itemize}
\item $i_{1}\geq i$ and $k_{1}\leq k$,
\item $\succsim_{(i,k)}^{R_1}\;\subseteq\;\succsim_{(i,k)}^{R_{1}^{'}}$,
%\item $({R_2},{R_{2}^{'}})\in\left\{(N,N),(N,P),(P,P)\right\}$,
\end{itemize}
then
%$$\succsim^{R_1,R_2}_{(i,k),{\cal D^{'}}}\;\subseteq \;\succsim^{R_1^{'},R_2^{'}}_{(i_1,k_1),{\cal D^{'}}}.$$
$$\succsim^{R_1,R_2}_{(i,k),{\cal D^{'}}}\;\subseteq \;\succsim^{R_1^{'},R_2^{'}}_{(i_1,k_1),{\cal D^{'}}} \;\; \mbox{where} \;\; 
R_{2}^{'}=\left\{
\begin{array}{lll}
N \;\mbox{or} \; P & \mbox{if} & R_{2}=N\\
\\
P & \mbox{if} & R_{2}=P.\\ 
\end{array}
\right.
$$

\end{prop}
\begin{proof}
See Appendix.
\end{proof}

\noindent For example, considering a $2$-point interval decomposition and supposing that $a\succsim_{(1,1),{\cal D^{'}}}^{N,N}b$ holds, using this proposition we obtain: $a\succsim_{(1,1),{\cal D^{'}}}^{P,N}b$, $a\succsim_{(1,1),{\cal D^{'}}}^{N,P}b$, $a\succsim_{(1,1),{\cal D^{'}}}^{P,P}b$, $a\succsim_{(2,1),{\cal D^{'}}}^{N,N}b$, $a\succsim_{(2,1),{\cal D^{'}}}^{P,N}b$, $a\succsim_{(2,1),{\cal D^{'}}}^{N,P}b$ and $a\succsim_{(2,1),{\cal D^{'}}}^{P,P}b$.

Joining the weak, the strong and the normal preference relations in case of group decisions, we obtain the following proposition:

\begin{prop}\label{group_normal_index} For any $R_{1}\in\left\{P,N\right\},$ and for any ${\cal D^{'}}\subseteq{\cal D}:$
\begin{enumerate}
\item $\succsim_{\cal D^{'}}^{SN,R_1}\; \subseteq \; \succsim_{\cal D^{'}}^{N,R_1} \; \subseteq \; \succsim_{\cal D^{'}}^{WN,R_1}$, 
\item $\succsim_{\cal D^{'}}^{SP,R_1}\; \subseteq \; \succsim_{\cal D^{'}}^{P,R_1} \; \subseteq \; \succsim_{\cal D^{'}}^{WP,R_1}$.
\end{enumerate}
\end{prop}
\begin{proof}
See Appendix.
\end{proof}
Considering together results of Propositions \ref{Prop_relations}, \ref{group_property} and \ref{group_normal_index}, we obtain the dependencies presented in Figure \ref{Group_Imprecise}.

\begin{figure}[ht!]
\caption{Group preference relations in case of imprecise evaluations\label{Group_Imprecise}}
\begin{center}
\includegraphics[scale=0.47]{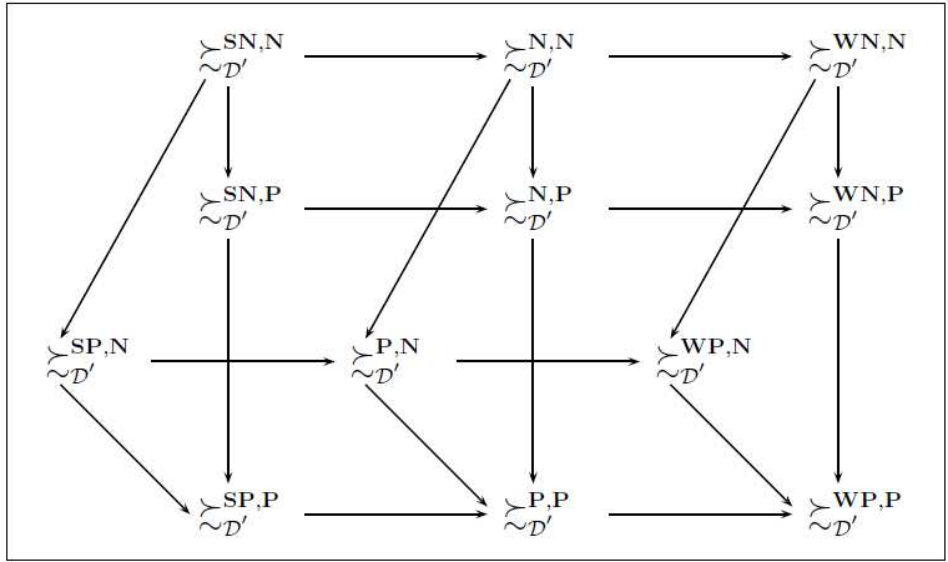}
\end{center}
\end{figure}

\vspace{0,2cm}
We have seen in Proposition \ref{Prop_relations} that for all $a,b\in A$ the following completeness property holds:

$$ a\succsim_{(i,k)}^{N}b \;\;\mbox{or}\;\; b\succsim_{(k,i)}^{P}a, \;\;\mbox{with}\;\;i,k\in\left\{1,\ldots,n\right\}.$$

In case of group decisions, the two following completeness properties hold:

\begin{prop}\label{Compl_prop}
For all $a,b\in A$, and for all $i,k\in\left\{1,\ldots,n\right\}$, we have:
\begin{enumerate}
\item $a\succsim_{(i,k),{\cal D^{'}}}^{N,N}b$ or $b\succsim_{(k,i),{\cal D^{'}}}^{P,P}a$,
\item $a\succsim_{(i,k),{\cal D^{'}}}^{N,P}b$ or $b\succsim_{(k,i),{\cal D^{'}}}^{P,N}a$.
\end{enumerate}
\end{prop}
\begin{proof}
See Appendix.
\end{proof}

For example, in Figure \ref{Group_Imprecise} we observe the following completeness properties, for all $a,b\in A$:
\begin{itemize}
\item $a\succsim^{SN,N}_{\cal D^{'}}b$ or $b\succsim^{WP,P}_{\cal D^{'}}a$,
\item $a\succsim^{N,N}_{\cal D^{'}}b$ or $b\succsim^{P,P}_{\cal D^{'}}a$,
\item $a\succsim^{WN,N}_{\cal D^{'}}b$ or $b\succsim^{SP,P}_{\cal D^{'}}a$,
\item $a\succsim^{SN,P}_{\cal D^{'}}b$ or $b\succsim^{WP,N}_{\cal D^{'}}a$,
\item $a\succsim^{N,P}_{\cal D^{'}}b$ or $b\succsim^{P,N}_{\cal D^{'}}a$,
\item $a\succsim^{WN,P}_{\cal D^{'}}b$ or $b\succsim^{SP,N}_{\cal D^{'}}a$.
\end{itemize}

\begin{prop}\label{group_transitivity}
Given $a,b,c\in A$, $i,k,i_1,k_1\in\left\{1,\ldots,n\right\}$, $R_1,R_2\in\left\{P,N\right\}$ and ${\cal D^{'}}\subseteq{\cal D}$ such that $k\geq i_1$:
\begin{enumerate}
\item if $a\succsim_{(i,k),{\cal D^{'}}}^{N,R_1}b,$ $b\succsim_{(i_1,k_1),{\cal D^{'}}}^{N,R_2}c$, then $a\succsim_{(r,s),{\cal D^{'}}}^{N,\overline{R}}c,$ 
\item if $a\succsim_{(i,k),{\cal D^{'}}}^{N,R_1}b,$ $b\succsim_{(i_1,k_1),{\cal D^{'}}}^{P,R_2}c$, then $a\succsim_{(r,s),{\cal D^{'}}}^{P,\overline{R}}c,$ 
\item if $a\succsim_{(i,k),{\cal D^{'}}}^{P,R_1}b,$ $b\succsim_{(i_1,k_1),{\cal D^{'}}}^{N,R_2}c$, then $a\succsim_{(r,s),{\cal D^{'}}}^{P,\overline{R}}c,$ 
\end{enumerate}
where $r,s\in\left\{1,\ldots,n\right\}$ such that $r\geq i$, $s\leq k_1$ and $\overline{R}=
\left\{
\begin{array}{lll}
R_1 & \mbox{if} & R_1=R_2=N,\\
P   & \mbox{if} & R_1\neq R_2.
\end{array}
\right.
$
\end{prop}
\begin{proof}
See Appendix.
\end{proof}
Let us note that in the above proposition, the property of transitivity does not hold when the two relations in the condition part are possibly true for at least one DM, i.e. when $a\succsim^{R_1}_{(i,k)}b$ is true for at least one DM and $b\succsim^{R_1}_{(i_1,k_1)}c$ is true for at least one DM. In fact, in this case, we cannot conclude that there exists at least one DM for whom $a\succsim^{R_1}_{(i,k)}c$, because the DMs of the first and of the second relations could be different.

%%%%%%%%%%%%%%%%%%%%%%%%%%%%%%%%%%%%%%%%%%%%%%%%%%%%%%%%%%%%%%%%%%%%%%%%%%%%%%%%%%%%%
\section{Further extensions of ROR in case of imprecise evaluations}\label{Fur_ext}%%
%%%%%%%%%%%%%%%%%%%%%%%%%%%%%%%%%%%%%%%%%%%%%%%%%%%%%%%%%%%%%%%%%%%%%%%%%%%%%%%%%%%%%

%%%%%%%%%%%%%%%%%%%%%%%%%%%%%%%%%%%%
\noindent \textbf{Inconsistency}\\%%
%%%%%%%%%%%%%%%%%%%%%%%%%%%%%%%%%%%%
We have seen in section \ref{ROR}, that the first step of ROR is to check if there exists at least one value function compatible with the preference information provided by the DM. In fact, it is possible that the information provided by the DM has some inconsistencies that do not permit to find a compatible value function. In this case, the DM, together with the analyst, can decide to continue the study while accepting the existence of these inconsistencies, or look for a minimal set of constraints responsible for this infeasibility, and remove them from the linear program. The procedures used to find a minimal set of constraints responsible for the infeasibility can be found in \cite{mousseau2003resolving}.

%%%%%%%%%%%%%%%%%%%%%%%%%%%%%%%%%%
\vspace{0,5cm}%%%%%%%%%%%%%%%%%%%%
\noindent \textbf{Credibility}\\%%
%%%%%%%%%%%%%%%%%%%%%%%%%%%%%%%%%%
ROR methods permit to elicit incrementally the preferences by the DM, assigning them a different degree of credibility. The idea of considering a sequence of preference information pieces ordered according to their credibility has been introduced in \cite{greco2008ordinal}. More formally, the preference information given by the DM is represented by a chain of embedded preference relations  $\succsim_{1}\subseteq\ldots\subseteq\succsim_{h}$, where for all $r,s=1,\ldots,h,$ with $r<s$, the preference $\succsim_r$ is more credible than $\succsim_s$. If for any $t=1,\ldots,h$, we denote by $E_{t}$ the set of constraints obtained from $\succsim_t$, and by ${\cal U}_t$ the sets of value functions compatible with the preference information of $\succsim_{t},$ then we have $E_1\subseteq\ldots\subseteq E_h$ and ${\cal U}_1\supseteq\ldots\supseteq{\cal U}_h$, and consequently $\succsim_{1}^{N}\subseteq\ldots\subseteq\succsim_h^{N}$, and $\succsim_{1}^{P}\supseteq\ldots\supseteq\succsim_{h}^{P},$ that is the smaller the credibility of the considered preference relation $\succsim_{t}$, the richer the necessary preference relation $\succsim^N_t$ and the poorer the possible preference relation $\succsim^P_t$. In case of imprecise evaluations, considering the same $E_{t}$ and ${\cal U}_t$, we will have for all $i,k\in\left\{1,\ldots,n\right\}$, $\succsim_{(i,k),1}^{N}\subseteq\ldots\subseteq\succsim_{(i,k),h}^{N}$, and $\succsim_{(i,k),1}^{P}\supseteq\ldots\supseteq\succsim_{(i,k),h}^{P},$ which means that also in this case the smaller the credibility of the considered preference relation $\succsim_t$, the richer the necessary preference relation $\succsim^{N}_{(i,k),t}$, and the poorer the possible preference relation $\succsim^{P}_{(i,k),t}.$

%%%%%%%%%%%%%%%%%%%%%%%%%%%%%%%%%%%%%%%%%%%%%%%
\vspace{0,5cm}%%%%%%%%%%%%%%%%%%%%%%%%%%%%%%%%%
\noindent \textbf{Extreme ranking analysis}\\%%
%%%%%%%%%%%%%%%%%%%%%%%%%%%%%%%%%%%%%%%%%%%%%%%
Necessary and possible preference relations give information regarding pairs of alternatives. However, it is also interesting to analyse information related to the whole set of alternatives in terms of the best and the worst ranking position assigned to each alternative by the compatible value functions. This constitutes the extreme ranking analysis introduced in \cite{Promethee2010}. In case of $n$-point imprecise evaluations, the extreme ranking analysis can be performed for each preference relation considered. That is, for all $i,k\in\left\{1,\ldots,n\right\}$ we could find a ranking using the necessary preference relation $\succsim_{(i,k)}^{N}$ and the possible preference relation $\succsim_{(i,k)}^{P}$.

%%%%%%%%%%%%%%%%%%%%%%%%%%%%%%%%%%%%%%
\vspace{0,5cm}%%%%%%%%%%%%%%%%%%%%%%%%
\noindent \textbf{Sorting problem}\\%%
%%%%%%%%%%%%%%%%%%%%%%%%%%%%%%%%%%%%%%
Ranking and choice problems are based on pairwise comparisons of alternatives and therefore they can be dealt with possible and necessary preference relations. Sorting relies instead on the alternative's intrinsic value and not on the comparison to others. Therefore, sorting problems need specific methods. Within ROR, UTADIS$^{GMS}$ \cite{greco2010multiple} has been proposed to deal with sorting problem as follows. Given a set of pre-defined classes $C_1,C_2,\ldots,C_p$ ordered from the worst to the best, the DM gives preference information in terms of exemplary assignments of reference alternatives to intervals of considered classes, such that $a_{\ast}\rightarrow\left[C_{L^{DM}}(a_{\ast}),C_{R^{DM}}(a_{\ast})\right]$, with $L^{DM}\leq R^{DM}$, means that reference alternative $a_{\ast}$ can be assigned to one of the classes between $C_{L^{DM}}(a_{\ast})$ and $C_{R^{DM}}(a_{\ast})$. Denoting by  $A^{R}\subseteq A$ the set of reference alternatives assigned by the DM, we say that a value function $U$ is compatible if 
\begin{equation}\label{CVF}
\forall a_{\ast},b_{\ast}\in A^{R}, L^{DM}(a_{\ast})>R^{DM}(b_{\ast})\Rightarrow U(a_{\ast})>U(b_{\ast}).
\end{equation}

Denoting by ${\cal U}$ the set of compatible value functions, we have that each $U\in{\cal U}$ assigns an alternative $a\in A$ to an interval of classes $\left[L^{U}(a),R^{U}(a)\right]$ where

$$
L^{U}(a)=\max\left(\left\{1\right\}\cup\left\{L^{DM}(a_{\ast}): U(a_{\ast})\leq U(a),a_{\ast}\in A^{R}\right\}\right),
$$

$$
R^{U}(a)=\min\left(\left\{p\right\}\cup\left\{R^{DM}(a_{\ast}): U(a_{\ast})\geq U(a),a_{\ast}\in A^{R}\right\}\right).
$$

Within ROR, considering the whole set of compatible value functions, for each $a\in A$ one can define the possible assignment $C^{P}(a)$ and the necessary assignment $C^{N}(a)$ as follows: 
\begin{itemize}
\item $C^{P}(a)=\left[L^{\cal U}_{P}(a),R^{\cal U}_{P}(a)\right]=\underset{U\in{\cal U}}{\bigcup}\left[L^{U}(a),R^{U}(a)\right]$,
\item $C^{N}(a)=\left[L^{\cal U}_{N}(a),R^{\cal U}_{N}(a)\right]=\underset{U\in{\cal U}}{\bigcap}\left[L^{U}(a),R^{U}(a)\right]$.
\end{itemize}

\noindent In case of imprecise evaluations, for each compatible value function $U\in{\cal U}$, and for all $i,k\in\left\{1,\ldots,n\right\}$ we consider: 

$$
L^{U}_{(i,k)}(a)=\max\left(\left\{1\right\}\cup\left\{L^{DM}(a_{\ast}): U(a^{(i)})\geq U(a_{\ast}^{(k)}),a_{\ast}\in A^{R}\right\}\right)
$$

$$
R^{U}_{(i,k)}(a)=\min\left(\left\{p\right\}\cup\left\{R^{DM}(a_{\ast}): U(a^{(i)})\leq U(a_{\ast}^{(k)}),a_{\ast}\in A^{R}\right\}\right).
$$

\begin{itemize}
\item $C^{P}_{(i,k)}(a)=\left[L^{\cal U}_{(i,k)}(a),R^{\cal U}_{(i,k)}(a)\right]=\underset{U\in{\cal U}}{\bigcup}\left[L_{(i,k)}^{U}(a),R_{(i,k)}^{U}(a)\right]$,
\item $C^{N}_{(i,k)}(a)=\left[L^{\cal U}_{(i,k)}(a),R^{\cal U}_{(i,k)}(a)\right]=\underset{U\in{\cal U}}{\bigcap}\left[L_{(i,k)}^{U}(a),R_{(i,k)}^{U}(a)\right]$.
\end{itemize}

\noindent In this way, we can check to which classes each alternative $a$ can be assigned for each couple $(i,k)$, where $i,k\in\left\{1,\ldots,n\right\}$.

%%%%%%%%%%%%%%%%%%%%%%%%%%%%%%%%%%%%%%
\vspace{0,5cm}%%%%%%%%%%%%%%%%%%%%%%%%
\noindent \textbf{Interacting criteria}\\%%
%%%%%%%%%%%%%%%%%%%%%%%%%%%%%%%%%%%%%%
UTA$^{GMS}$, UTADIS$^{GMS}$ and GRIP use an additive value function as preference model. This model is among the most popular ones because it has the advantage of being easily manageable, and, moreover, it has a very sound axiomatic basis (see, e.g., \cite{Keeney76,wakker1989additive}). However, the additive value function is not able to represent \emph{interactions} among criteria. For example, consider evaluation of cars using such criteria as maximum speed, acceleration and price. In this case, there may exist a negative interaction (\emph{negative synergy}) between maximum speed and acceleration
because a car with a high maximum speed also has a good acceleration, so, even if each of these two criteria is very
important for a DM who likes sport cars, their joint impact on reinforcement of preference of a more speedy and better accelerating
car over a less speedy and worse accelerating car will be smaller than a simple addition of the two impacts corresponding to each of
the two criteria considered separately in validation of this preference relation. In the same decision problem, there may exist a
positive interaction (\emph{positive synergy}) between maximum speed and price because a car with a high maximum speed usually also has a
high price, and thus a car with a high maximum speed and relatively low price is very much appreciated. Thus, the comprehensive impact
of these two criteria on the strength of preference of a more speedy and cheaper car over a less speedy and more expensive car is greater
than the impact of the two criteria considered separately in validation of this preference relation. To handle the interactions among criteria, one can consider \textit{non-additive integrals}, such as Choquet integral \cite{choquet1953theory} and Sugeno integral \cite{sugeno1974theory} (for a comprehensive survey on the use of non-additive integrals in MCDA see \cite{grabisch2000fuzzy}). Robust ordinal regression has also been applied to the Choquet integral 
\cite{angilella2010non}. The non-additive integrals suffer, however, from  limitations within MCDA (see \cite{roy-rairo-2009,roy-rairo-2007}); in particular, they need that the evaluations of alternatives on all criteria are expressed on the same scale. This means that in order to apply a non-additive integral it is
necessary, for example, to estimate if the maximum speed of 200 km/h is worth the price of 35,000\euro{}. To deal with this problem a new robust ordinal regression method has been proposed: $UTA^{GMS}$--\emph{INT} \cite{greco2008ordinal}. The preference model used by $UTA^{GMS}$--\emph{INT} is a general additive value function augmented by two components corresponding to ``bonus'' and ``penalty'' values for pairs of positively or negatively interacting criteria, respectively. When calculating the value of a particular alternative, a bonus is added to the additive component of the value function if a given pair of criteria is in a positive synergy for evaluations of this alternative on the two criteria, or the penalty is subtracted from the additive component of the value function if a given pair of criteria is in a negative synergy for evaluations of this alternative on the two criteria. The specific formulation of the value function in $UTA^{GMS}$--\emph{INT} permits to deal with criteria having heterogeneous scales, without encoding evaluations of alternatives on a common scale. 

%%%%%%%%%%%%%%%%%%%%%%%%%%%%%%%%%%%%%%%%%%%%%%%%
\vspace{0,5cm}%%%%%%%%%%%%%%%%%%%%%%%%%%%%%%%%%%
\noindent \textbf{Methodological extensions}\\%%
%%%%%%%%%%%%%%%%%%%%%%%%%%%%%%%%%%%%%%%%%%%%%%%%
In this paper we have shown how to compute for all $i,k\in\left\{1,\ldots,n\right\}$, and for each pair of alternatives $a,b\in A$, the dominance relation $a\Delta^{(i,k)}b$ and the preference relation $a\succsim^{(i,k)}b$, using the same $i$-th indicator for all criteria when considering alternative $a$, and the same $k$-th indicator for all criteria when considering alternative $b$. A possible extension of the described method, could consist in taking into account different indicators for each criterion; for example, one could consider for each alternative indicator $g_{j}^{i}$ for criterion $g_{j}$, indicator $g_{h}^{k}$ for criterion $g_{h}$, and so on, defining other dominance and preference relations, as well as other fictitious alternatives involving evaluations on the selected indicators. For example, let us suppose that evaluations of alternatives $a$ and $b$ are as shown in Table \ref{eval}:

\begin{table}[!ht]
\caption{Evaluation table of alternatives $a$ and $b$}\label{eval}
\begin{center}
\begin{tabular}{|cccc|}
\hline
\rule[-3mm]{-1mm}{0.8cm}
alternative & $g_{1}(\cdot)$ & $g_{2}(\cdot)$ & $g_{3}(\cdot)$ \\
\hline
\rule[-3mm]{-1mm}{0.8cm}
$a$ & [10,15,17] & [25,40,50] & [23,47,56] \\
\rule[-3mm]{-1mm}{0.8cm}
$b$ & [11,23,45] & [5,20,65] & [2,82,90] \\
\hline
\end{tabular}
\end{center}
\end{table}

\noindent In this case, if one would like to consider dominance and preference relations with respect to indicator $g_{1}^{(3)}$ for criterion $g_{1}$, indicator $g_{2}^{(1)}$ for criterion $g_{2}$ and indicator $g_{3}^{(2)}$ for criterion $g_{3}$, then the fictitious alternatives $\underline{a}$ and $\underline{b}$ would get evaluations as shown in Table \ref{fict_alt_1}:

\begin{table}[!ht]
\caption{Evaluation table of fictitious alternatives $\overline{a}$ and $\overline{b}$}\label{fict_alt_1}
\begin{center}
\begin{tabular}{|cccc|}
\hline
\rule[-3mm]{-1mm}{0.8cm}
alternative & $g_{1}(\cdot)$ & $g_{2}(\cdot)$ & $g_{3}(\cdot)$ \\
\hline
\rule[-3mm]{-1mm}{0.8cm}
$\underline{a}$ & [17,17,17] & [25,25,25] & [47,47,47] \\
\rule[-3mm]{-1mm}{0.8cm}
$\underline{b}$ & [45,45,45] & [5,5,5] & [82,82,82] \\
\hline
\end{tabular}
\end{center}
\end{table}

\noindent Another meaningful extension regards missing evaluations. For example, supposing that alternative $a$ has missing evaluations on criterion $g$, one can consider for this criterion a $2$-point interval $g(a)=\left[\alpha_{g},\beta_{g}\right]$, where $\alpha_{g}$ and $\beta_{g}$ are respectively the worst and the best values an alternative could have on criterion $g$. In this context, we could only consider the weak and the strong dominance and preference relations, reducing also the other criteria to $2$-point intervals. For example, suppose that evaluations of alternatives $a$ and $b$ are as shown in Table \ref{miss_eval}:

\begin{table}[!ht]
\caption{Evaluation table of alternatives $a$ and $b$ in case of missing evaluations}\label{miss_eval}
\begin{center}
\begin{tabular}{|cccc|}
\hline
\rule[-3mm]{-1mm}{0.8cm}
alternative & $g_1(\cdot)$ & $g_2(\cdot)$ & $g_3(\cdot)$\\
\hline
\rule[-3mm]{-1mm}{0.8cm}
$a$ & [10,15,17] & [25,40,50] & $\cdot$ \\
\rule[-3mm]{-1mm}{0.8cm}
$b$ & [11,23,45] & $\cdot$ & [2,82,90] \\
\hline
\end{tabular}
\end{center}
\end{table}

\noindent Alternative $a$ miss evaluation on criterion $g_{3}$, and $b$ miss evaluation on criterion $g_{2}.$ In this case, supposing that the range of evaluations on criterion $g_{2}$ is $\left[10,70\right]$, while the range of evaluations on criterion $g_{3}$ is $\left[1,100\right]$, we will consider the fictitious alternatives $a^{(W)}$, $a^{(B)}$, $b^{(W)}$ and $b^{(B)}$ representing the worst and the best realizations of alternatives $a$ and $b$ respectively, having $2$-point interval evaluations as shown in Table \ref{miss_eval_2}:

\begin{table}[!ht]
\caption{Evaluation table of fictitious alternatives $a^{(W)}$, $a^{(B)}$, $b^{(W)}$, $b^{(B)}$}\label{miss_eval_2}
\begin{center}
\begin{tabular}{|cccc|}
\hline
\rule[-3mm]{-1mm}{0.8cm}
alternative & $g_1(\cdot)$ & $g_2(\cdot)$ & $g_3(\cdot)$\\
\hline
\rule[-3mm]{-1mm}{0.8cm}
$a^{(W)}$ & [10,10] & [25,25] & [1,1] \\
\rule[-3mm]{-1mm}{0.8cm}
$a^{(B)}$ & [17,17] & [50,50] & [100,100] \\
\rule[-3mm]{-1mm}{0.8cm}
$b^{(W)}$ & [11,11] & [10,10] & [2,2] \\
\rule[-3mm]{-1mm}{0.8cm}
$b^{(B)}$ & [45,45] & [70,70] & [90,90] \\
\hline
\end{tabular}
\end{center}
\end{table}

\noindent Consequently, all known evaluations of $a$ and $b$, which are characterized by $3$-point intervals, are reduced to $2$-point intervals when we consider the weak and the strong dominance and preference relations for the pair $(a,b)$.

\section{A didactic example}
In this didactic example, for the sake of simplicity, we consider only the ``classic'' preference relations $\succsim^{N}$ and $\succsim^{P}$, the ``strong'' preference relations $\succsim^{SN}$ and $\succsim^{SP}$ and the ``weak'' preference relations $\succsim^{WN}$ and $\succsim^{WP}$. Let us suppose a High School has to give a scholarship; for this reason, the Dean has to choose a laureate among 10 students being the best in particular classes of the school. In order to cope with this problem, the Dean decides to use a multicriteria approach, considering each student evaluated on three subjects: Mathematics (Mat), Physics (Phy) and Computer Science (Com). Each subject is thus an evaluation criterion with an ordinal scale constituted from five levels ordered from the worst to the best: ``Very Bad, Bad, Medium, Good and Very Good.'' Differently from the last years, the Dean has a new problem, because some students have imprecise evaluations on some criteria. The students' evaluations are shown in Table \ref{EVALUATIONS}.

\begin{table}[!htp]
\caption{Evaluations of students on three criteria}\label{EVALUATIONS}
\begin{center}
\resizebox{10cm}{!}{
\begin{tabular}{c|ccc}
student$\backslash$subject & Mat & Phy & Com  \\
\hline
$\mathbf{A}$ & Medium & Very Good & Very Good \\ 
$\mathbf{B}$ & [Good,Very Good] & [Very Bad,Medium] & [Bad,Good] \\
$\mathbf{C}$ & [Bad,Very Good] & Good & [Medium,Good] \\
$\mathbf{D}$ & [Good,Very Good] & [Medium,Good] & [Medium,Good] \\
$\mathbf{E}$ & Very Good & [Very Bad,Good] & [Medium,Good] \\
$\mathbf{F}$ & [Very Bad,Good] & [Bad,Medium] & [Bad,Medium] \\
$\mathbf{H}$ & [Medium,Good] & [Medium,Good] & [Medium,Good] \\
$\mathbf{I}$ & Very Good & [Medium,Very Good] & Bad \\
$\mathbf{L}$ & [Very Bad,Bad] & [Bad,Medium] & [Very Bad,Medium]  \\
$\mathbf{M}$ & [Very Bad,Bad] & [Good,Very Good] & Very Good \\
\end{tabular}
}
\end{center}
\end{table}

In Table $\ref{EVALUATIONS}$, we see that $\mathbf{A}$ is the only alternative having precise evaluations on all criteria, while all other alternatives have imprecise evaluations; for example Mat($\mathbf{F}$)=[Very Bad,Good] means that alternative $\mathbf{B}$ can assume the following evaluations ``Very Bad, Bad, Medium and Good'' on Mathematics, that is the only evaluation he certainly will not have is ``Very Good.'' From Table \ref{EVALUATIONS} and Definition $\ref{dominance}$, we obtain Figures \ref{Weak_Dom}, \ref{Norm_Dom} and \ref{Strong_Dom} showing the weak, normal and strong dominance relations, respectively.

\begin{figure}[ht!]
\begin{center}
\includegraphics[scale=0.6]{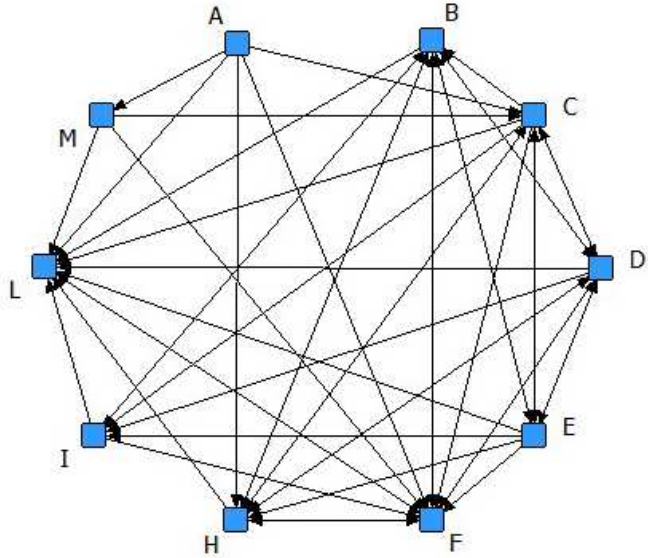}
\end{center}
\caption{Weak dominance relation in the set of students\label{Weak_Dom}}
\end{figure}

\begin{figure}[ht!]
\begin{center}
\includegraphics[scale=0.6]{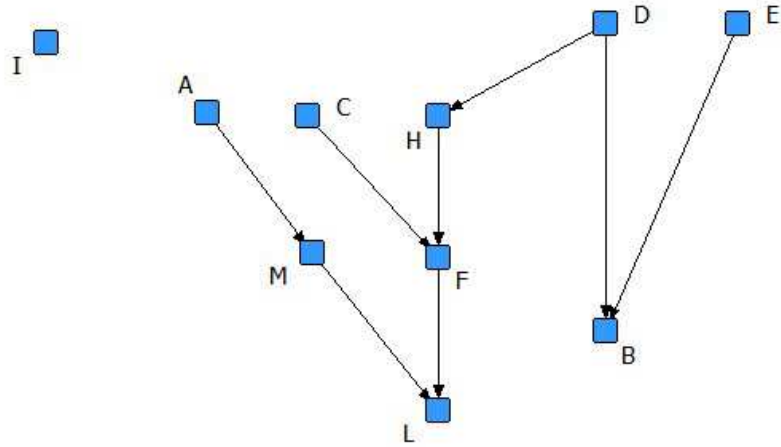}
\end{center}
\caption{Classic dominance relation in the set of students\label{Norm_Dom}}
\end{figure}

\begin{figure}[ht!]
\begin{center}
\includegraphics[scale=0.6]{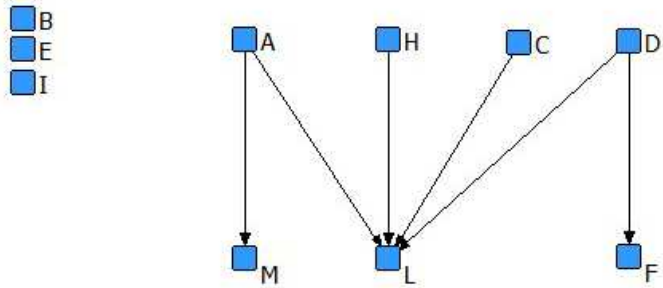}
\end{center}
\caption{Strong dominance relation in the set of students\label{Strong_Dom}}
\end{figure}

The Dean provides, moreover, preference information regarding few alternatives he is confident on. 
\begin{itemize}
\item At first, he says that ``student $\mathbf{M}$ is preferred to student $\mathbf{D}$''; this information is translated into the constraint $U(\mathbf{M})>U(\mathbf{D})$ that we shall call $C_{1}$; adding $C_{1}$ to the linear programming constraints $\left(E^{A^{R'}}\right)$, and after solving the corresponding optimization problems, we obtain preference relations shown in Figures \ref{Nor_FPI} and \ref{Strong_FPI}; in these  Figures, thick red arrows represent new information.

\begin{figure}[ht!]
\begin{center}
\includegraphics[scale=0.6]{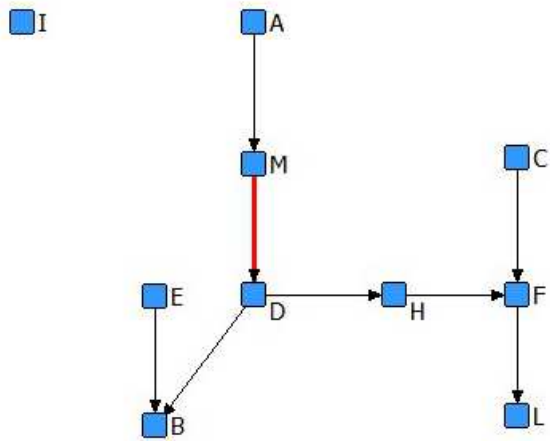}
\end{center}
\caption{Classic necessary preference relation obtained for the first piece of preference information\label{Nor_FPI}}
\end{figure}

\begin{figure}[ht!]
\begin{center}
\includegraphics[scale=0.65]{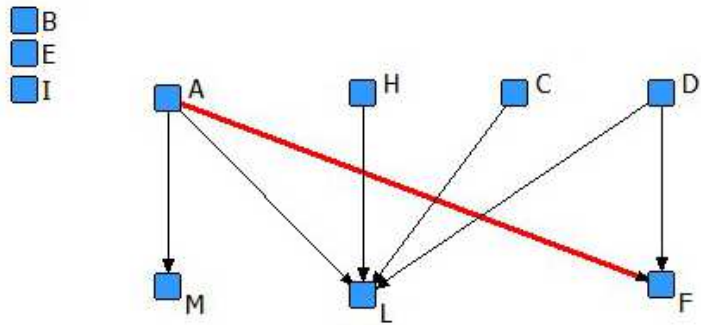}
\end{center}
\caption{Strong necessary preference relation obtained for the first piece of preference information\label{Strong_FPI}}
\end{figure}

\item Then, the Dean gives information expressing intensities of preference: ``Student $\mathbf{M}$ is preferred to student $\mathbf{I}$ more than student $\mathbf{C}$ is preferred to student $\mathbf{H}$''; adding the constraint $U(\mathbf{M})-U(\mathbf{I})>U(\mathbf{C})-U(\mathbf{H})$ (that we shall call $C_{2}$) to the linear programming constraints $\left(E^{A^{R'}}\right)$ and $C_{1}$, and solving the corresponding optimization problems, we obtain the preference relations shown in Figures \ref{Nor_SPI} and \ref{Strong_SPI}.

\begin{figure}[ht!]
\begin{center}
\includegraphics[scale=0.6]{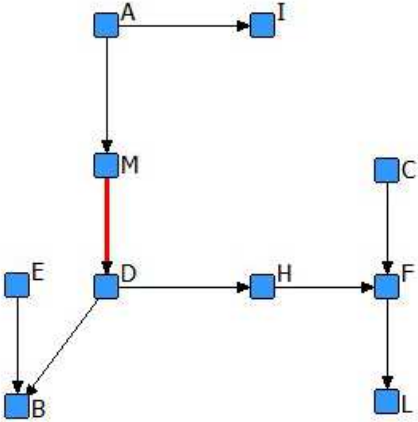}
\end{center}
\caption{Classic necessary preference relation obtained for the second piece of preference information\label{Nor_SPI}}
\end{figure}

\begin{figure}[ht!]
\begin{center}
\includegraphics[scale=0.6]{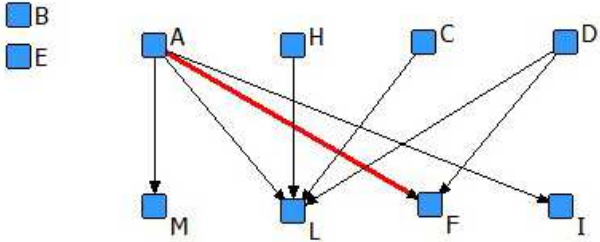}
\end{center}
\caption{Strong necessary preference relation obtained for the second piece of preference information\label{Strong_SPI}}
\end{figure}

\item Finally, the Dean expresses an indifference relation: ``Student $\mathbf{C}$ and student $\mathbf{M}$ are indifferent;'' adding the constraint $U(\mathbf{C})=U(\mathbf{M})$ to the linear programming constraints constituted by $\left(E^{A^{R'}}\right), C_{1}$ and $C_{2}$, and solving the corresponding optimization problems, we obtain Figure \ref{Last_pref}.

\begin{figure}[ht!]
\begin{center}
\includegraphics[scale=0.6]{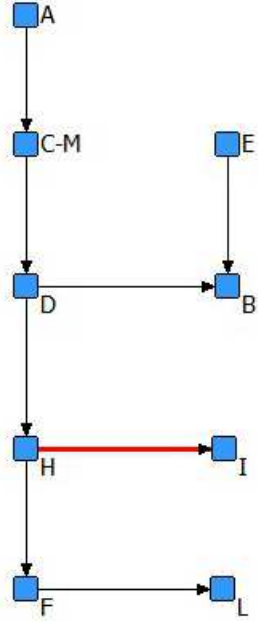}
\end{center}
\caption{Classic necessary preference relation obtained for the third piece of preference information\label{Last_pref}}
\end{figure}
\end{itemize}

At this stage, looking at the classic necessary preference relation, the Dean could conclude that, given his preferences, the best candidates for the scholarship are students $\mathbf{A}$ and $\mathbf{E}$.

%%%%%%%%%%%%%%%%%%%%%%%%%%%%%%
\section{Conclusions}%%%%%%%%%
%%%%%%%%%%%%%%%%%%%%%%%%%%%%%%
In this paper, we dealt with one of the most important issues of MCDA, that is the imprecise evaluations of alternatives. The possible sources of this imprecision are, for example, lack of data, imprecise measurement or intangible criteria. Many authors have studied different types of imprecision regarding weights of criteria, utility functions or probabilities about the different states of the world. In our approach, we are considering alternatives having imprecise evaluations on particular criteria, i.e., regarding criterion $g_j$, alternative $a$ has an evaluation in the $n$-point interval $\left[g_j^{1}(a),\ldots,g_j^{n}(a)\right]$, where $g_{j}^{i}$ is called indicator. In this way, if $m$ is the number of evaluation criteria,  each alternative is represented by $n\times m$ values corresponding to the different evaluations obtained by the alternative from the above indicators. In order to take into account these imprecise evaluations, first we have considered for all $i,k\in\left\{1,\ldots,n\right\}$ one dominance relation $\Delta^{(i,k)}$ for which $a\Delta^{(i,k)}b$ if the value obtained by $a$ from the $i$-th indicator is at least as good as the value obtained by $b$ from the $k$-th indicator for each criterion $g_j$. In order to compare the alternatives having imprecise evaluations, we have used an additive value function obtained by adding up $n\times m$ marginal value functions, one for each considered indicator. For each alternative $a\in A$ and for each $i\in\left\{1,\ldots,n\right\}$, we have considered a fictitious alternative $a^{(i)}$ having precise evaluations on the considered criteria, equal to the performance of $a$ with respect to each of the $i$-th indicators, i.e. $g_{j}^{k}(a^{(i)})=g_{j}^{i}(a),$ $k=1,\ldots,n,$ and $j=1,\ldots,m$. Using this information, in the ROR context, we have built two binary preference relations, one necessary and one possible, for all $i,k\in\left\{1,\ldots,n\right\}$:
\begin{itemize}
\item alternative $a$ is $(i,k)$-possibly preferred to alternative $b$ if $a^{(i)}$ is at least as good as $b^{(k)}$ for at least one compatible value function,
\item alternative $a$ is $(i,k)$-necessarily preferred to alternative $b$ if $a^{(i)}$ is at least as good as $b^{(k)}$ for all compatible value functions. 
\end{itemize}
We have also presented a generalization of these preference relations in case of group decisions, showing several properties that hold in this particular case. When presenting the methodology for dealing with imprecise evaluations, we have supposed that criteria are characterized by $n$-intervals, where, for the sake of simplicity, $n$ is equal for all criteria, but in the section devoted to generalizations, we also explained how to proceed in case $n$ depends on the considered criteria, and in case of missing evaluations.\\
Remark, finally, that $n$-point imprecise evaluations can be considered as a particular case of the hierarchy of criteria, which has been studied in \cite{CGSHierarchy}. In fact, the evaluation of alternative $a$ on criterion $g_{j}$ depends on the evaluations of $a$ on the $n$ indicators $g_{j}^{i}$ that could be considered as subcriteria descending from criterion $g_{j}$ and situated at the level below. 

%%%%%%%%%%%%%%%%%%%%%%%%%%%
\section*{Acknowledgment}%%
%%%%%%%%%%%%%%%%%%%%%%%%%%%
The third author wishes to acknowledge financial support from the Polish National Science Centre, grant no. N N519 441939.

%\bibliographystyle{plain}
%\bibliography{General}

%%%%%%%%%%%%%%%%%%%%%%%%%%%%%%%
\section*{Appendix}%%%%%%%%%%%%
%%%%%%%%%%%%%%%%%%%%%%%%%%%%%%%
%%%%%%%%%%%%%%%%%%%%%%%%%%%%%%%%%%%%%%%%%%%%%%%%%%%%%%%%%%%%
\textbf{Proof of Proposition \ref{Dominance Proposition}}%%%
%%%%%%%%%%%%%%%%%%%%%%%%%%%%%%%%%%%%%%%%%%%%%%%%%%%%%%%%%%%%

\vspace{0,1cm}
\begin{enumerate}
\item Let $a\in A,$ and $i,k\in\left\{1,\ldots,n\right\}$ such that $i\geq k$; this implies that $g_{j}^{i}(a)\geq g_{j}^{k}(a),\; \forall j=1,\ldots,m,$ and thus  $a\Delta^{(i,k)}a$. Therefore $\Delta^{(i,k)}$ is reflexive. 

\item Let us suppose that $a,b,c\in A$ such that $a\Delta^{(i,k)}b$, $b\Delta^{(i,k)}c$ and $i\leq k$ with $i,k\in\left\{1,\ldots,n\right\}$;  
$$
\left\{
\begin{array}{lll}
a\Delta^{(i,k)}b & \Leftrightarrow & g_{j}^{i}(a)\geq g_{j}^{k}(b), \forall j\in J,\\
\\
b\Delta^{(i,k)}c & \Leftrightarrow & g_{j}^{i}(b)\geq g_{j}^{k}(c), \forall j\in J,\\
\end{array}
\right.
\Rightarrow 
$$
\noindent being $i\leq k$ and combining with above expressions, we have: 
$$
g_{j}^{i}(a)\geq g_{j}^{k}(b)\geq g_{j}^{i}(b)\geq g_{j}^{k}(c), \forall j=1,\ldots,m \Rightarrow g_{j}^{i}(a)\geq g_{j}^{k}(c), \forall j\in J \Leftrightarrow a\Delta^{(i,k)}c.
$$
\noindent Thus $\Delta^{(i,k)}$ is transitive.
\item It follows by points \ref{dom_reflexivity} and \ref{dom_transitivity} of this Proposition.
\item Let be $a,b\in A$ and $i,k,r,s\in\left\{1,\ldots,n\right\}$ such that $a\Delta^{(i,k)}b$, $r\geq i$ and $k\geq s$. Then 
$$
\left\{
\begin{array}{lll}
a\Delta^{(i,k)}b & \Leftrightarrow & g_{j}^{i}(a)\geq g_{j}^{k}(b), \forall j\in J,\\
\\
r\geq i & \Leftrightarrow & g_{j}^{r}(a)\geq g_{j}^{i}(a), \forall j\in J, \\
\\
k\geq s & \Leftrightarrow &  g_{j}^{k}(b)\geq g_{j}^{s}(b), \forall j\in J. 
\end{array}
\right.
$$
\noindent From this it follows that:
$$
g_{j}^{r}(a)\geq g_{j}^{i}(a)\geq g_{j}^{k}(b)\geq g_{j}^{s}(b), \;\forall j\in J\Rightarrow g_{j}^{r}(a)\geq g_{j}^{s}(b),\;\forall j\in J\Leftrightarrow a{\Delta}^{(r,s)}b.
$$
\item Let be $a,b,c\in A$, and $i,k,i_1,k_1\in\left\{1,\ldots,n\right\}$ such that $a\Delta^{(i,k)}b$, $b\Delta^{(i_1,k_1)}c$ and $k\geq i_1$. Then we have:
$$
\left\{
\begin{array}{lll}
a\Delta^{(i,k)}b & \Leftrightarrow & g_{j}^{i}(a)\geq g_{j}^{k}(b), \forall j\in J,\\
\\
b\Delta^{(i_1,k_1)}c & \Leftrightarrow & g_{j}^{i_1}(b)\geq g_{j}^{k_1}(c), \forall j\in J,\\
\\
k\geq i_1 & \Leftrightarrow & g_{j}^{k}(b)\geq g_{j}^{i_1}(b), \forall j\in J.
\end{array}
\right.
$$
\noindent From this it follows that:
$$
g_{j}^{i}(a)\geq g_{j}^{k}(b)\geq g_{j}^{i_1}(b)\geq g_{j}^{k_1}(c),\; \forall j\in J\Rightarrow g_{j}^{i}(a)\geq g_{j}^{k_1}(c),\;\forall j\in J \Leftrightarrow a\Delta^{(i,k_1)}c.
$$ By point \ref{Dom_incl} of Proposition \ref{Dominance Proposition}, if $r,s\in\left\{1,\ldots,n\right\}$ such that $r\geq i$ and $s\leq k_1$ then $a\Delta^{(r,s)}c.$ 
\item We have said that $\Delta=\displaystyle\cap_{i=1}^{n}\Delta^{(i,i)}$; being $\Delta^{(i,i)}$ a partial preorder for point \ref{part_preorder} of this Proposition, and being the intersection of partial preorders a partial preoder, then $\Delta$ is a partial preorder.
\item Let be $a,b,c\in A$, and $i,k\in\left\{1,\ldots,n\right\}$, such that $a\Delta^{(i,k)}b$, $b\Delta c$, $s\geq i$ and $k\geq t$. Then we have:
$$
\left\{
\begin{array}{lll}
a\Delta^{(i,k)}b & \Leftrightarrow & g_{j}^{i}(a)\geq g_{j}^{k}(b), \forall j\in J,\\
\\
b\Delta c & \Leftrightarrow & g_{j}^{r}(b)\geq g_{j}^{r}(c), \forall j\in J,\forall r=1,\ldots,n,\\
\\
s\geq i & \Leftrightarrow & g_{j}^{s}(a)\geq g_{j}^{i}(a), \forall j\in J,\\
\\
k\geq t & \Leftrightarrow & g_{j}^{k}(c)\geq g_{j}^{t}(c), \forall j\in J,\\
\end{array}
\right.
$$
\noindent From this it follows that:
$$
g_{j}^{s}(a)\geq g_{j}^{i}(a)\geq g_{j}^{k}(b)\geq g_{j}^{k}(c)\geq g_{j}^{t}(c),\; \forall j\in J\Rightarrow g_{j}^{s}(a)\geq g_{j}^{t}(c),\;\forall j\in J \Leftrightarrow a\Delta^{(s,t)}c.
$$
\item Let be $a,b,c\in A$, and $i,k\in\left\{1,\ldots,n\right\}$, such that $a\Delta b$, $b\Delta^{(i,k)} c$, $s\geq i$ and $k\geq t$. Then we have:
$$
\left\{
\begin{array}{lll}
a\Delta b & \Leftrightarrow & g_{j}^{r}(a)\geq g_{j}^{r}(b), \forall j\in J,\forall r=1,\ldots,n,\\
\\
b\Delta^{(i,k)}c & \Leftrightarrow & g_{j}^{i}(b)\geq g_{j}^{k}(c), \forall j\in J,\\
\\
s\geq i & \Leftrightarrow & g_{j}^{s}(a)\geq g_{j}^{i}(a), \forall j\in J,\\
\\
k\geq t & \Leftrightarrow & g_{j}^{k}(c)\geq g_{j}^{t}(c), \forall j\in J,\\
\end{array}
\right.
$$
\noindent From this it follows that:
$$
g_{j}^{s}(a)\geq g_{j}^{i}(a)\geq g_{j}^{i}(b)\geq g_{j}^{k}(c)\geq g_{j}^{t}(c),\; \forall j\in J\Rightarrow g_{j}^{s}(a)\geq g_{j}^{t}(c),\;\forall j\in J \Leftrightarrow a\Delta^{(s,t)}c.
$$
\end{enumerate}

%%%%%%%%%%%%%%%%%%%%%%%%%%%%%%%%%%%%%%%%%%%%%%%%%%%%%%%%%%%%
\vspace{0,2cm}%%%%%%%%%%%%%%%%%%%%%%%%%%%%%%%%%%%%%%%%%%%%%%
\noindent\textbf{Proof of Proposition \ref{salvo_prop}}%%%%%
%%%%%%%%%%%%%%%%%%%%%%%%%%%%%%%%%%%%%%%%%%%%%%%%%%%%%%%%%%%%
%%%%%%%%%%%%%%%%%%%%%%%%%%%%%%%%%%%%%%%%%%%%%%%%%%%%%%%%%%%%

\vspace{0,1cm}
\begin{enumerate}
\item For all $i\in\left\{1,\ldots,n\right\}$, by point \ref{Dom_incl} of Proposition \ref{Dominance Proposition}, we have $\Delta^{(1,n)}\;\subseteq\;\Delta^{(i,i)}$ because, $i\geq 1$ and $i\leq n$; from this follows that $\Delta^{(1,n)}\;\subseteq\;\cap_{i=1}^{n}\Delta^{(i,i)}=\Delta$ and this proves the first part of the proposition. \\
For the same reason, $\forall i\in\left\{1,\ldots,n\right\}$ we have $\Delta^{(i,i)}\;\subseteq\;\Delta^{(n,1)}$ because $n\geq i$ and $1\leq i$. From this it follows that $\Delta=\cap_{i=1}^{n}\Delta^{(i,i)}\;\subseteq\;\Delta^{(n,1)}$; this proves the second part of the proposition.
\item For all $i,k\in\left\{1,\ldots,n\right\}$, by point \ref{Dom_incl} of Proposition \ref{Dominance Proposition}, being $1\leq i$ and $k\leq n$ we have $\Delta^{(1,n)}\subseteq \Delta^{(i,k)}$ and $\Delta^{(i,k)}\subseteq\Delta^{(n,1)}.$ In this way we obtain the thesis.
\end{enumerate}

%%%%%%%%%%%%%%%%%%%%%%%%%%%%%%%%%%%%%%%%%%%%%%%%%%%%%%%%%%%%%%%%
\vspace{0,2cm}%%%%%%%%%%%%%%%%%%%%%%%%%%%%%%%%%%%%%%%%%%%%%%%%%%
\noindent\textbf{Proof of Proposition \ref{Prop_relations}}%%%%%
%%%%%%%%%%%%%%%%%%%%%%%%%%%%%%%%%%%%%%%%%%%%%%%%%%%%%%%%%%%%%%%%
%%%%%%%%%%%%%%%%%%%%%%%%%%%%%%%%%%%%%%%%%%%%%%%%%%%%%%%%%%%%%%%%

\vspace{0,1cm}
\begin{enumerate}
\item Let $a\in A$ and $i,k\in\left\{1,\ldots,n\right\}$. For definition, fictitious alternatives $a^{(i)}$ and $a^{(k)}$ are such that $\forall j\in J$, and $\forall r\in\left\{1,\ldots,n\right\},$ $g_{j}^{r}\left(a^{(i)}\right)=g_{j}^{i}(a)$, and $g_{j}^{r}\left(a^{(k)}\right)=g_{j}^{k}(a)$. Being $i\geq k$, $\forall r\in\left\{1,\ldots,n\right\}$ and $\forall j\in J$, $g_{j}^{r}(a^{(i)})\geq g_{j}^{r}(a^{(k)})$, and using the monotonicity of marginal value functions $u_{j,r}(\cdot)$ we obtain $\forall r\in\left\{1,\ldots,n\right\}$, $\forall j\in J$ $u_{j,r}(g_{j}^{r}(a^{(i)}))\geq u_{j,r}(g_{j}^{r}(a^{(k)}))$; adding up with respect to $j$ and $r$ we obtain the thesis.
\item We have seen that: 
$$
U(a^{(1)})=\sum_{j=1}^{m}\left[\sum_{i=1}^{n}u_{j,i}(g_{j}^{1}(a))\right],\;\;\;U(a)=\sum_{j=1}^{m}\left[\sum_{i=1}^{n}u_{j,i}(g_{j}^{i}(a))\right],\;\;\; U(a^{(n)})=\sum_{j=1}^{m}\left[\sum_{i=1}^{n}u_{j,i}(g_{j}^{n}(a))\right].
$$
\noindent $\forall j\in J,$ and $\forall i\in\left\{1,\ldots,n\right\}$, being $g_{j}^{1}(a)\leq g_{j}^{i}(a)\leq g_{j}^{n}(a)$ and by monotonicity of marginal value functions $u_{j,i}(\cdot)$, we obtain:
$$
u_{j,i}(g_{j}^{1}(a))\leq u_{j,i}(g_{j}^{i}(a))\leq u_{j,i}(g_{j}^{n}(a))
$$
\noindent and therefore adding up with respect to $j$ and $i$ we obtain the thesis, that is 
$$U(a^{(1)})\leq U(a)\leq U(a^{(n)}).$$
\item See \cite{greco2008ordinal}.
%$$ a\succsim^{N}b \Leftrightarrow U(a)\geq U(b), \forall U\in{\cal U} \Rightarrow \exists U\in{\cal U}: U(a)\geq U(b)\Rightarrow a\succsim^{P}b.$$
\item $\forall a,b\in A$, $\forall i,k\in\left\{1,\ldots,n\right\}$, if $a^{(i)}$ is at least as good as $b^{(k)}$ for all compatible value functions $(a\succsim_{(i,k)}^{N}b)$, then there exists at least one compatible value function for which $a^{(i)}$ is at least as good as $b^{(k)}$ $(a\succsim_{(i,k)}^{P}b)$.
\item See \cite{greco2008ordinal}.
\item See \cite{greco2008ordinal}.
\item See \cite{greco2008ordinal}.
\item Let $a,b\in A$, such that $a\Delta b$. This implies that $g_{j}^{i}(a)\geq g_{j}^{i}(b), \forall j\in J, \forall i=1,\ldots,n.$ We know that, for all $U\in{\cal U}$: 
$$
U(a)=\sum_{j=1}^{m}\left[\sum_{i=1}^{n}u_{j,i}(g_{j}^{i}(a))\right].
$$
\noindent By monotonicity of marginal value functions $u_{j,i}(\cdot)$, we have that $\forall j\in J,$ and $\forall i=1,\ldots,n,$ $u_{j,i}(g_{j}^{i}(a))\geq u_{j,i}(g_{j}^{i}(b))$ and adding up with respect to indices $j$ and $i$ we obtain $U(a)\geq U(b)$ for all compatible value functions, then we obtain the thesis.
\item See \cite{greco2008ordinal}.
\item See \cite{greco2008ordinal}.
\item It follows by point \ref{inclusion_dominance_nec} of this Proposition, and by transitivity of $\succsim^{N}.$
\item It follows by point \ref{inclusion_dominance_nec} of this Proposition, and by transitivity of $\succsim^{N}.$
\item It follows by points \ref{inclusion_dominance_nec} and \ref{necc_pos} of this Proposition.
\item It follows by points \ref{inclusion_dominance_nec} and \ref{pos_necc} of this Proposition.
\item Let $a,b\in A$, and $i,k\in\left\{1,\ldots,n\right\},$ such that $a\Delta^{(i,k)}b$. This implies that $g_{j}^{i}(a)\geq g_{j}^{k}(b), \forall j\in J.$ We know that, for all $U\in{\cal U}$: 
$$
U(a^{(i)})=\sum_{j=1}^{m}\left[\sum_{r=1}^{n}u_{j,r}(g_{j}^{i}(a))\right],\;\;\;U(b^{(k)})=\sum_{j=1}^{m}\left[\sum_{r=1}^{n}u_{j,r}(g_{j}^{k}(b))\right].
$$
\noindent By the monotonicity of marginal value functions $u_{j,i}(\cdot)$ we have that $\forall j\in J,$ and $\forall i=1,\ldots,n,$ $u_{j,r}(g_{j}^{i}(a))\geq u_{j,r}(g_{j}^{k}(b))$ and adding up with respect to indices $j$ and $r$ we obtain the thesis.
\item Let $a,b,c\in A$ and $i,r\in\left\{1,\ldots,n\right\}$ such that $a\succsim^{N}_{(i,n)}b$, $b\succsim^{N}c$ and $r\geq i$. Then we have:
$$
\left\{
\begin{array}{lll}
a\succsim^{N}_{(i,n)}b & \Leftrightarrow & U(a^{(i)})\geq U(b^{(n)}), \forall U\in{\cal U},\\
\\
b\succsim^{N}c & \Leftrightarrow & U(b)\geq U(c), \forall U\in{\cal U},\\
\\
r\geq i & \Rightarrow & U(a^{(r)})\geq U(a^{(i)}), \forall U\in{\cal U}.
\end{array}
\right.
$$
\noindent It follows that, for all $U\in{\cal U}$, $U(a^{(r)})\geq U(a^{(i)})\geq U(b^{(n)})\geq U(b)\geq U(c)\geq U(c^{(1)})$ where $U(b^{(n)})\geq U(b)$ and $U(c)\geq U(c^{(1)})$ hold by point \ref{chain} of this Proposition. Thus, for all $U\in{\cal U}$ we obtain $U(a^{(r)})\geq U(c^{(1)})$, and therefore $a\succsim_{(r,1)}^{N}c.$ \\
\item Let $a,b,c\in A$ and $k,r\in\left\{1,\ldots,n\right\}$ such that $a\succsim^{N}b$, $b\succsim_{(1,k)}^{N}c$ and $r\leq k$. Then we have:
$$
\left\{
\begin{array}{lll}
a\succsim^{N}b & \Leftrightarrow & U(a)\geq U(b), \forall U\in{\cal U},\\
\\
b\succsim_{(1,k)}^{N}c & \Leftrightarrow & U(b^{(1)})\geq U(c^{(k)}), \forall U\in{\cal U},\\
\\
r\leq k & \Rightarrow & U(c^{(r)})\leq U(c^{(k)}), \forall U\in{\cal U}.
\end{array}
\right.
$$
\noindent It follows that, for all $U\in{\cal U}$, $U(a^{(n)})\geq U(a)\geq U(b)\geq U(b^{(1)})\geq U(c^{(k)})\geq U(c^{(r)})$ where $U(a^{(n)})\geq U(a)$ and $U(b)\geq U(b^{(1)})$ hold by point \ref{chain} of this Proposition. Thus, for all $U\in{\cal U}$ we obtain $U(a^{(n)})\geq U(c^{(r)})$, and therefore $a\succsim_{(n,r)}^{N}c.$
\item Let $a,b,c\in A$ and $i,r\in\left\{1,\ldots,n\right\}$ such that $a\succsim_{(i,n)}^{P}b$, $b\succsim^{N}c$ and $r\geq i$. Then we have:
$$
\left\{
\begin{array}{lll}
a\succsim_{(i,n)}^{P}b & \Leftrightarrow & \exists U\in{\cal U}: U(a^{(i)})\geq U(b^{(n)}),\\
\\
b\succsim^{N}c & \Leftrightarrow & U(b)\geq U(c), \forall U\in{\cal U},\\
\\
r\geq i & \Rightarrow & U(a^{(r)})\geq U(a^{(i)}), \forall U\in{\cal U}.
\end{array}
\right.
$$
\noindent It follows that, there exists $U\in{\cal U}$ such that $U(a^{(r)})\geq U(a^{(i)})\geq U(b^{(n)})\geq U(b)\geq U(c)\geq U(c^{(1)})$ where $U(b^{(n)})\geq U(b)$ and $U(c)\geq U(c^{(1)})$ hold by point \ref{chain} of this Proposition. Thus, there exists $U\in{\cal U}$ such that $U(a^{(r)})\geq U(c^{(1)})$, and therefore $a\succsim_{(r,1)}^{P}c.$
\item Let $a,b,c\in A$ and $k,r\in\left\{1,\ldots,n\right\}$ such that $a\succsim^{N}b$, $b\succsim_{(1,k)}^{P}c$ and $r\leq k$. Then we have:
$$
\left\{
\begin{array}{lll}
a\succsim^{N}b & \Leftrightarrow & U(a)\geq U(b), \forall U\in{\cal U},\\
\\
b\succsim_{(1,k)}^{P}c & \Leftrightarrow & \exists U\in{\cal U}: U(b^{(1)})\geq U(c^{(k)}),\\
\\
r\leq k & \Rightarrow & U(c^{(r)})\leq U(c^{(k)}), \forall U\in{\cal U}.
\end{array}
\right.
$$
\noindent It follows that, there exists $U\in{\cal U}$ such that $U(a^{(n)})\geq U(a)\geq U(b)\geq U(b^{(1)})\geq U(c^{(k)})\geq U(c^{(r)})$ where $U(a^{(n)})\geq U(a)$ and $U(b)\geq U(b^{(1)})$ hold by point \ref{chain} of this Proposition. Thus, there exists $U\in{\cal U}$ such that $U(a^{(n)})\geq U(c^{(r)})$, and therefore $a\succsim_{(n,r)}^{P}c.$
\item Let $a,b,c\in A$ and $i,r\in\left\{1,\ldots,n\right\}$ such that $a\succsim_{(i,n)}^{N}b$, $b\succsim^{P}c$ and $r\geq i$. Then we have:
$$
\left\{
\begin{array}{lll}
a\succsim_{(i,n)}^{N}b & \Leftrightarrow & U(a^{(i)})\geq U(b^{(n)}), \forall U\in{\cal U},\\
\\
b\succsim^{P}c & \Leftrightarrow & \exists U\in{\cal U}: U(b)\geq U(c),\\
\\
r\geq i & \Rightarrow & U(a^{(r)})\geq U(a^{(i)}), \forall U\in{\cal U}.
\end{array}
\right.
$$
\noindent It follows that, there exists $U\in{\cal U}$ such that $U(a^{(r)})\geq U(a^{(i)})\geq U(b^{(n)})\geq U(b)\geq U(c)\geq U(c^{(1)})$ where $U(b^{(n)})\geq U(b)$ and $U(c)\geq U(c^{(1)})$ hold by point \ref{chain} of this Proposition. Thus, there exists $U\in{\cal U}$ such that $U(a^{(r)})\geq U(c^{(1)})$, and therefore $a\succsim_{(r,1)}^{P}c.$
\item Let $a,b,c\in A$ and $k,r\in\left\{1,\ldots,n\right\}$ such that $a\succsim^{P}b$, $b\succsim_{(1,k)}^{N}c$ and $r\leq k$. Then we have:
$$
\left\{
\begin{array}{lll}
a\succsim^{P}b & \Leftrightarrow & \exists U\in{\cal U}: U(a)\geq U(b),\\
\\
b\succsim_{(1,k)}^{N}c & \Leftrightarrow & U(b^{(1)})\geq U(c^{(k)}), \forall U\in{\cal U},\\
\\
r\leq k & \Rightarrow & U(c^{(r)})\leq U(c^{(k)}), \forall U\in{\cal U}.
\end{array}
\right.
$$
\noindent It follows that, there exists $U\in{\cal U}$ such that $U(a^{(n)})\geq U(a)\geq U(b)\geq U(b^{(1)})\geq U(c^{(k)})\geq U(c^{(r)})$ where $U(a^{(n)})\geq U(a)$ and $U(b)\geq U(b^{(1)})$ hold by point \ref{chain} of this Proposition. Thus, there exists $U\in{\cal U}$ such that $U(a^{(n)})\geq U(c^{(r)})$, and therefore $a\succsim_{(n,r)}^{P}c.$
\item It follows from points \ref{inclusion_dominance} and \ref{nec_ind_nec} of this Proposition.
\item It follows from points \ref{inclusion_dominance} and \ref{nec_nec_ind} of this Proposition.
\item It follows from points \ref{inclusion_dominance} and \ref{nec_ind_pos} of this Proposition.
\item It follows from points \ref{inclusion_dominance} and \ref{pos_nec_ind} of this Proposition.
\item It follows from points \ref{inclusion_dominance_nec} and \ref{nec_nec_ind} of this Proposition.
\item It follows from points \ref{inclusion_dominance_nec} and \ref{nec_ind_nec} of this Proposition.
\item It follows from points \ref{inclusion_dominance_nec} and \ref{nec_pos_ind} of this Proposition.
\item It follows from points \ref{inclusion_dominance_nec} and \ref{pos_ind_nec} of this Proposition.
\item It follows by point \ref{monotonicity} of this Proposition.
\item Let $a,b,c\in A$ and $i,k\in\left\{1,\ldots,n\right\}$, such that $a\succsim^{N}_{(i,k)}b,$ $b\succsim^{N}_{(i,k)}c$ and $i\leq k$. Then we have:
$$
\left\{
\begin{array}{lll}
a\succsim^{N}_{(i,k)}b & \Leftrightarrow & U(a^{(i)})\geq U(b^{(k)}), \forall U\in{\cal U},\\
\\
b\succsim^{N}_{(i,k)}c & \Leftrightarrow & U(b^{(i)})\geq U(c^{(k)}), \forall U\in{\cal U},\\
\end{array}
\right.
$$
\noindent By point \ref{monotonicity} of this Proposition, $U(a^{(i)})\geq U(b^{(k)})\geq U(b^{(i)})\geq U(c^{(k)})$, $\forall U\in{\cal U}$; thus $U(a^{(i)})\geq U(c^{(k)})$, $\forall U\in{\cal U}$ and therefore $a\succsim_{(i,k)}^{N}c.$
\item It follows by points \ref{reflexivity} and \ref{transitivity} of this Proposition because $i\geq i$ implies $\succsim_{(i,i)}^{N}$ is reflexive, and $i\leq i$ implies $\succsim_{(i,i)}^{N}$ is transitive.
\item Let $a,b\in A$, and $i,k\in\left\{1,\ldots,n\right\}$, such that $a\not\succsim_{(i,k)}^{N}b.$ This means that $\exists U\in{\cal U}: U(a^{(i)})<U(b^{(k)}).$ Therefore, $b\succsim_{(k,i)}^{P}(a).$
\item Let $a,b\in A$, and $i,k\in\left\{1,\ldots,n\right\}$ with $i\geq k$ such that $a\not\succsim_{(i,k)}^{P}b$. This means that for all $U\in{\cal U}$, $U(b^{(k)})>U(a^{(i)})$. Being $i\geq k$, and by point \ref{monotonicity} of this Proposition, we obtain that for all $U\in{\cal U},$ $U(b^{(i)})\geq U(b^{(k)})>U(a^{(i)})\geq U(a^{(k)})$; thus for all $U\in{\cal U}$, $U(b^{(i)})> U(a^{(k)})$, therefore $b\succsim_{(i,k)}^{N}a$ implying $b\succsim_{(i,k)}^{P}a$ by point \ref{inclus_nec_pos_n_point} of this Proposition. In this way $\succsim_{(i,k)}^{P}$ is strongly complete.\\
Let $a,b,c\in A$, $i,k\in\left\{1,\ldots,n\right\}$, such that $i\geq k$, $a\not\succsim_{(i,k)}^{P}b$ and $b\not\succsim_{(i,k)}^{P}c.$ Then we have:
$$
\left\{
\begin{array}{lll}
a\not\succsim_{(i,k)}^{P}b & \Leftrightarrow & U(a^{(i)})<U(b^{(k)}), \forall U\in{\cal U},\\
\\
b\not\succsim_{(i,k)}^{P}c & \Leftrightarrow & U(b^{(i)})<U(c^{(k)}), \forall U\in{\cal U},\\
\end{array}
\right.
$$
\noindent From this and using point \ref{monotonicity} of this Proposition it follows that 
$$U(a^{(i)})<U(b^{(k)})\leq U(b^{(i)})<U(c^{(k)}),\; \forall U\in{\cal U}\Rightarrow U(a^{(i)})< U(c^{(k)}), \;\forall U\in{\cal U}\Leftrightarrow a\not\succsim_{(i,k)}^{P}c.
$$ This proves that $\succsim_{(i,k)}^{P}$ is negatively transitive.
\item Let $a,b\in A$, and $i,k,i_1,k_1\in\left\{1,\ldots,n\right\}$, such that $i_1\geq i$, $k_1\leq k$ and $a\succsim_{(i,k)}^{N}b$. Then, by point \ref{monotonicity} of this Proposition, we have:
$$
\left\{
\begin{array}{lll}
a\succsim_{(i,k)}^{N}b & \Leftrightarrow & U(a^{(i)})\geq U(b^{(k)}), \forall U\in{\cal U},\\
\\
i_1\geq i & \Rightarrow & U(a^{(i_1)})\geq U(a^{(i)}), \forall U\in{\cal U},\\
\\
k_1\leq k & \Rightarrow & U(b^{(k_1)})\leq U(b^{(k)}), \forall U\in{\cal U}.\\
\end{array}
\right.
$$
\noindent Thus: 
$$
U(a^{(i_1)})\geq U(a^{(i)})\geq U(b^{(k)})\geq U(b^{(k_1)}), \forall U\in{\cal U}\Rightarrow U(a^{(i_1)})\geq U(b^{(k_1)}), \forall U\in{\cal U}\Leftrightarrow a\succsim_{(i_1,k_1)}^{N}b.
$$
\item Let $a,b\in A$, and $i,k,i_1,k_1\in\left\{1,\ldots,n\right\}$ such that $i_1\geq i$, $k_1\leq k$ and $a\succsim_{(i,k)}^{P}b$. Then, by point \ref{monotonicity} of this Proposition, we have:
$$
\left\{
\begin{array}{lll}
a\succsim_{(i,k)}^{P}b & \Leftrightarrow & \exists U\in{\cal U}: U(a^{(i)})\geq U(b^{(k)}),\\
\\
i_1\geq i & \Rightarrow & U(a^{(i_1)})\geq U(a^{(i)}), \forall U\in{\cal U},\\
\\
k_1\leq k & \Rightarrow & U(b^{(k_1)})\leq U(b^{(k)}), \forall U\in{\cal U}.\\
\end{array}
\right.
$$
\noindent Thus: 
$$
\exists U\in{\cal U}: U(a^{(i_1)})\geq U(a^{(i)})\geq U(b^{(k)})\geq U(b^{(k_1)})\Rightarrow \exists U\in{\cal U}: U(a^{(i_1)})\geq U(b^{(k_1)})\Leftrightarrow a\succsim_{(i_1,k_1)}^{P}b.
$$
\item For all $i,k=1,\ldots,n$, it is true that $i\geq 1$ and $k\leq n$; thus by point \ref{nec_inclusion} of this Proposition, we have $\succsim_{(1,n)}^{N}\;\subseteq\;\succsim_{(i,k)}^{N}$; at the same time, being $n\geq i$ and $1\leq k$, by point \ref{nec_inclusion} of this Proposition, we have $\succsim_{(i,k)}^{N}\;\subseteq\;\succsim_{(n,1)}^{N}.$ In this way we obtain the thesis, that is $\succsim_{(1,n)}^{N}\;\subseteq\;\succsim_{(i,k)}^{N}\;\subseteq\;\succsim_{(n,1)}^{N}.$
\item For all $i,k=1,\ldots,n$, it is true that $i\geq 1$ and $k\leq n$; thus by point \ref{pos_inclusion} of this Proposition, we have $\succsim_{(1,n)}^{P}\;\subseteq\;\succsim_{(i,k)}^{P}$; at the same time, being $n\geq i$ and $1\leq k$, by point \ref{pos_inclusion} of this Proposition, we have $\succsim_{(i,k)}^{P}\;\subseteq\;\succsim_{(n,1)}^{P}.$ In this way we obtain the thesis, that is $\succsim_{(1,n)}^{P}\;\subseteq\;\succsim_{(i,k)}^{P}\;\subseteq\;\succsim_{(n,1)}^{P}.$
\item Let $a,b\in A$ such that $a\succsim_{(1,n)}^{N}b.$ This means that $U(a^{(1)})\geq U(b^{(n)}),$ $\forall U\in{\cal U}$. By point \ref{chain} of this Proposition, we have that $U(a)\geq U(a^{(1)})\geq U(b^{(n)})\geq U(b),$ $\forall U\in{\cal U}$, and thus we obtain $U(a)\geq U(b)$, $\forall U\in{\cal U},$ that is $a\succsim^{N}b.$ In this way we proved that $\succsim_{(1,n)}^{N}\;\subseteq\;\succsim^{N}$. \\
Analogously, $a\succsim^{N}b$ means that $U(a)\geq U(b),$ $\forall U\in{\cal U}$; by point \ref{chain} of this Proposition we obtain $U(a^{(n)})\geq U(a)\geq U(b)\geq U(b^{(1)})$, $\forall U\in{\cal U}$, and thus we have $U(a^{(n)})\geq U(b^{(1)})$, $\forall U\in{\cal U},$ that is $a\succsim_{(n,1)}^{N}b.$ In this way we proved that $\succsim^{N}\;\subseteq\;\succsim_{(n,1)}^{N}.$ 
\item Let $a,b\in A$ such that $a\succsim_{(1,n)}^{P}b.$ This means that $\exists U\in{\cal U}: U(a^{(1)})\geq U(b^{(n)})$. By point \ref{chain} of this Proposition we have:
$$
U(a)\geq U(a^{(1)})\geq U(b^{(n)})\geq U(b)\Rightarrow U(a)\geq U(b) \Leftrightarrow a\succsim^{P}b.
$$ In this way we proved that $\succsim_{(1,n)}^{P}\;\subseteq\;\succsim^{P}$. \\
Analogously, $a\succsim^{P}b$ means that $\exists U\in{\cal U}: U(a)\geq U(b)$; by point \ref{chain}. of this Proposition we obtain: 
$$
U(a^{(n)})\geq U(a)\geq U(b)\geq U(b^{(1)})\Rightarrow U(a^{(n)})\geq U(b^{(1)})\Leftrightarrow a\succsim_{(n,1)}^{P}b.
$$
In this way we proved that $\succsim^{P}\;\subseteq\;\succsim_{(n,1)}^{P}.$
\item Let $a,b,c\in A$ and $i,k,i_1,k_1\in\left\{1,\ldots,n\right\}$, such that $a\succsim_{(i,k)}^{N}b$, $b\succsim_{(i_1,k_1)}^{N}c$ and $k\geq i_1$. Then we have:
$$
\left\{
\begin{array}{lll}
a\succsim_{(i,k)}^{N}b & \Leftrightarrow & U(a^{(i)})\geq U(b^{(k)}), \forall U\in{\cal U},\\
\\
b\succsim_{(i_1,k_1)}^{N}c & \Leftrightarrow & U(b^{(i_1)})\geq U(c^{(k_1)}), \forall U\in{\cal U},\\
\\
k\geq i_1 & \Rightarrow & U(b^{(k)})\geq U(b^{(i_1)})
\end{array}
\right.
$$
\noindent From this it follows:
$$
U(a^{(i)})\geq U(b^{(k)})\geq U(b^{(i_1)})\geq U(c^{(k_1)}),\forall U\in{\cal U}\Rightarrow U(a^{(i)})\geq U(c^{(k_1)}), \forall U\in{\cal U}\Leftrightarrow a\succsim_{(i,k_1)}^{N}c.
$$ 
Being $a\succsim_{(i,k_1)}^{N}c$ and $r,s\in\left\{1,\ldots,n\right\}$: $r\geq i$ and $s\leq k_1$, by point \ref{nec_inclusion} of this Proposition we obtain $a\succsim_{(r,s)}^{N}c.$ 
\item Let $a,b,c\in A,$ and $i,k,i_1,k_1\in\left\{1,\ldots,n\right\}$, such that $a\succsim_{(i,k)}^{N}b,$ $b\succsim_{(i_1,k_1)}^{P}c,$ and $k\geq i_1.$ Then we have:
$$
\left\{
\begin{array}{lll}
a\succsim_{(i,k)}^{N}b & \Leftrightarrow & U(a^{(i)})\geq U(b^{(k)}), \forall U\in{\cal U},\\
\\
b\succsim_{(i_1,k_1)}^{P}c & \Leftrightarrow & \exists U\in{\cal U}: U(b^{(i_1)})\geq U(c^{(k_1)}),\\
\\
k\geq i_1 & \Rightarrow & U(b^{(k)})\geq U(b^{(i_1)}), \forall U\in{\cal U}.
\end{array}
\right.
$$
\noindent It follows that:
$$
\exists U\in{\cal U}: U(a^{(i)})\geq U(b^{(k)})\geq U(b^{(i_1)})\geq U(c^{(k_1)})\Rightarrow \exists U\in{\cal U}: U(a^{(i)})\geq U(c^{(k_1)})\Leftrightarrow a\succsim_{(i,k_1)}^{P}c.
$$ 
Being $r,s\in\left\{1,\ldots,n\right\}$: $r\geq i$ and $s\leq k_1$, by point \ref{pos_inclusion} of this Proposition we obtain $a\succsim_{(r,s)}^{P}c.$ 
\item Let $a,b,c\in A,$ $i,k,i_1,k_1\in\left\{1,\ldots,n\right\}$ such that $a\succsim_{(i,k)}^{P}b,$ $b\succsim_{(i_1,k_1)}^{N}c,$ and $k\geq i_1.$ We have that:
$$
\left\{
\begin{array}{lll}
a\succsim_{(i,k)}^{P}b & \Leftrightarrow & \exists U\in{\cal U}: U(a^{(i)})\geq U(b^{(k)}),\\
\\
b\succsim_{(i_1,k_1)}^{N}c & \Leftrightarrow & U(b^{(i_1)})\geq U(c^{(k_1)}), \forall U\in{\cal U},\\
\\
k\geq i_1 & \Rightarrow & U(b^{(k)})\geq U(b^{(i_1)}), \forall U\in{\cal U}. 
\end{array}
\right.
$$
\noindent From this it follows: 
$$
\exists U\in{\cal U}: U(a^{(i)})\geq U(b^{(k)})\geq U(b^{(i_1)})\geq U(c^{(k_1)})\Rightarrow \exists U\in{\cal U}: U(a^{(i)})\geq U(c^{(k_1)})\Leftrightarrow a\succsim_{(i,k_1)}^{P}c.
$$ 
Being $r,s\in\left\{1,\ldots,n\right\}$ such that $r\geq i$ and $s\leq k_1$, by point \ref{pos_inclusion} of this Proposition we obtain $a\succsim_{(r,s)}^{P}c.$ 
\item It follows from points \ref{inclusion_dominance} and \ref{nec_transitivity} of this Proposition.
\item It follows from points \ref{inclusion_dominance} and \ref{nec_transitivity} of this Proposition.
\item It follows from points \ref{inclusion_dominance} and \ref{necpos_transitivity} of this Proposition.
\item It follows from points \ref{inclusion_dominance} and \ref{posnec_transitivity} of this Proposition.
\end{enumerate}

%%%%%%%%%%%%%%%%%%%%%%%%%%%%%%%%%%%%%%%%%%%%%%%%%%%%%%%%%%%%
\vspace{0,2cm}%%%%%%%%%%%%%%%%%%%%%%%%%%%%%%%%%%%%%%%%%%%%%%
\noindent\textbf{Proof of Proposition \ref{picture_prop}}%%%
%%%%%%%%%%%%%%%%%%%%%%%%%%%%%%%%%%%%%%%%%%%%%%%%%%%%%%%%%%%%
\vspace{0,2cm}

\noindent Proof of all points of this Proposition are straightforward consequences of Proposition \ref{Prop_relations} except for the following points:
\begin{description}
\item[1.] Let $a,b,c\in A$ such that $a\succsim^{SN}b$ and $b\succsim^{N}c$; it follows that:
$$
\left\{
\begin{array}{lll}
a\succsim^{SN}b & \Leftrightarrow & U(a^{(1)})\geq U(b^{(n)}), \forall U\in{\cal U},\\
\\
b\succsim^{N}c & \Leftrightarrow & U(b)\geq U(c), \forall U\in{\cal U}.\\
\end{array}
\right.
$$
By point \ref{chain} of Proposition \ref{Prop_relations} we have that for all $U\in{\cal U}$, $U(a)\geq U(a^{(1)})$ and $U(b^{(n)})\geq U(b)$ thus, 
$$U(a)\geq U(a^{(1)})\geq U(b^{(n)})\geq U(b)\geq U(c), \;\forall U\in{\cal U}$$ 
and therefore $U(a)\geq U(c)$, for all $U\in{\cal U}$ obtaining the thesis.
\item[2.] Let $a,b,c\in A$ such that $a\succsim^{SN}b$ and $b\succsim^{P}c$; it follows that:
$$
\left\{
\begin{array}{lll}
a\succsim^{SN}b & \Leftrightarrow & U(a^{(1)})\geq U(b^{(n)}), \forall U\in{\cal U},\\
\\
b\succsim^{P}c & \Leftrightarrow & \exists U\in{\cal U}:  U(b)\geq U(c).\\
\end{array}
\right.
$$
By point \ref{chain} of Proposition \ref{Prop_relations} we have that for all $U\in{\cal U}$, $U(a)\geq U(a^{(1)})$ and $U(b^{(n)})\geq U(b)$ thus, 
$$\exists U\in{\cal U}: U(a)\geq U(a^{(1)})\geq U(b^{(n)})\geq U(b)\geq U(c),$$ 
and therefore there exists $U\in{\cal U}$ such that $U(a)\geq U(c)$, obtaining the thesis.
\item[3.] Let $a,b,c\in A$ such that $a\succsim^{N}b$ and $b\succsim^{SN}c$; it follows that:
$$
\left\{
\begin{array}{lll}
a\succsim^{N}b & \Leftrightarrow & U(a)\geq U(b), \forall U\in{\cal U},\\
\\
b\succsim^{SN}c & \Leftrightarrow & U(b^{(1)})\geq U(c^{(n)}), \forall U\in{\cal U}.\\
\end{array}
\right.
$$
By point \ref{chain} of Proposition \ref{Prop_relations} we have that for all $U\in{\cal U}$, $U(b)\geq U(b^{(1)})$ and $U(c^{(n)})\geq U(c)$ thus, 
$$U(a)\geq U(b)\geq U(b^{(1)})\geq U(c^{(n)})\geq U(c), \forall U\in{\cal U}$$ 
and therefore $U(a)\geq U(c)$ for all $U\in{\cal U}$, obtaining the thesis.
\item[4.] Let $a,b,c\in A$ such that $a\succsim^{N}b$ and $b\succsim^{SP}c$; it follows that:
$$
\left\{
\begin{array}{lll}
a\succsim^{N}b & \Leftrightarrow & U(a)\geq U(b), \forall U\in{\cal U},\\
\\
b\succsim^{SP}c & \Leftrightarrow & \exists U\in{\cal U}: U(b^{(1)})\geq U(c^{(n)}).\\
\end{array}
\right.
$$
By point \ref{chain} of Proposition \ref{Prop_relations} we have that for all $U\in{\cal U}$, $U(b)\geq U(b^{(1)})$ and $U(c^{(n)})\geq U(c)$ thus, 
$$\exists U\in{\cal U}: U(a)\geq U(b)\geq U(b^{(1)})\geq U(c^{(n)})\geq U(c),$$ 
and therefore there exists $U\in{\cal U}$ such that $U(a)\geq U(c)$, obtaining the thesis.
\item[5.] Let $a,b,c\in A$ such that $a\succsim^{SP}b$ and $b\succsim^{N}c$; it follows that:
$$
\left\{
\begin{array}{lll}
a\succsim^{SP}b & \Leftrightarrow & \exists U\in{\cal U}: U(a^{(1)})\geq U(b^{(n)}),\\
\\
b\succsim^{N}c & \Leftrightarrow & U(b)\geq U(c), \forall U\in{\cal U}.\\
\end{array}
\right.
$$
By point \ref{chain} of Proposition \ref{Prop_relations} we have that for all $U\in{\cal U}$, $U(a)\geq U(a^{(1)})$ and $U(b^{(n)})\geq U(b)$ thus, 
$$\exists U\in{\cal U}: U(a)\geq U(a^{(1)})\geq U(b^{(n)})\geq U(b)\geq U(c),$$ 
and therefore there exists $U\in{\cal U}$ such that $U(a)\geq U(c)$, obtaining the thesis.
\item[6.] Let $a,b,c\in A$ such that $a\succsim^{P}b$ and $b\succsim^{SN}c$; it follows that:
$$
\left\{
\begin{array}{lll}
a\succsim^{P}b & \Leftrightarrow & \exists U\in{\cal U}: U(a)\geq U(b),\\
\\
b\succsim^{SN}c & \Leftrightarrow & U(b^{(1)})\geq U(c^{(n)}), \forall U\in{\cal U},\\
\end{array}
\right.
$$
By point \ref{chain} of Proposition \ref{Prop_relations} we have that for all $U\in{\cal U}$, $U(b)\geq U(b^{(1)})$ and $U(c^{(n)})\geq U(c)$ thus, 
$$\exists U\in{\cal U}: U(a)\geq U(b)\geq U(b^{(1)})\geq U(c^{(n)})\geq U(c),$$ 
and therefore there exists $U\in{\cal U}$ such that $U(a)\geq U(c)$, obtaining the thesis.
\item[13.] Let $a,b,c\in A$ such that $a\Delta^{S} b$ and $b\succsim^{N}c$. By point \ref{inclusion_dominance} of Proposition \ref{Prop_relations}, it follows that $a\succsim^{SN}c$. Then:
$$
\left\{
\begin{array}{lll}
a\succsim^{SN}b & \Leftrightarrow & U(a^{(1)})\geq U(b^{(n)}), \forall U\in{\cal U},\\
\\
b\succsim^{N}c & \Leftrightarrow & U(b)\geq U(c), \forall U\in{\cal U},\\
\end{array}
\right.
$$
\noindent thus, for all $U\in{\cal U}$, we obtain $U(a)\geq U(a^{(1)})\geq U(b^{(n)})\geq U(b)\geq U(c)$, where $U(a)\geq U(a^{(1)})$ and $U(b^{(n)})\geq U(b)$ hold by point \ref{chain} of Proposition \ref{Prop_relations}. Therefore $U(a)\geq U(c)$ for all $U\in{\cal U}$ from which the thesis $a\succsim^{N}c.$
\item[16.] Let $a,b,c\in A$ such that $a\Delta^{S} b$ and $b\succsim^{P}c$. By point \ref{inclusion_dominance} of Proposition \ref{Prop_relations}, it follows that $a\succsim^{SN}c$. Then:
$$
\left\{
\begin{array}{lll}
a\succsim^{SN}b & \Leftrightarrow & U(a^{(1)})\geq U(b^{(n)}), \forall U\in{\cal U},\\
\\
b\succsim^{P}c & \Leftrightarrow & \exists U\in{\cal U}: U(b)\geq U(c),\\
\end{array}
\right.
$$
\noindent thus, there exists $U\in{\cal U}: U(a)\geq U(a^{(1)})\geq U(b^{(n)})\geq U(b)\geq U(c)$, where $U(a)\geq U(a^{(1)})$ and $U(b^{(n)})\geq U(b)$ hold by point \ref{chain} of Proposition \ref{Prop_relations}. Therefore there exists $U\in{\cal U}:$ $U(a)\geq U(c)$ from which the thesis $a\succsim^{P}c.$
\item[20.] It follows from point \ref{inclusion_dominance_nec} of Proposition \ref{Prop_relations} and from point \ref{N_SN} of this Proposition.
\item[21.] It follows from point \ref{inclusion_dominance_nec} of Proposition \ref{Prop_relations} and from point \ref{N_SP} of this Proposition.
\item[26.] It follows from point \ref{inclusion_dominance_nec} of Proposition \ref{Prop_relations} and from point \ref{SN_N} of this Proposition.
\item[31.] It follows from point \ref{inclusion_dominance} of Proposition \ref{Prop_relations} and from point \ref{N_SN} of this Proposition.
\item[36.] It follows from point \ref{inclusion_dominance_nec} of Proposition \ref{Prop_relations} and from point \ref{SP_N} of this Proposition.
\item[40.] It follows from point \ref{inclusion_dominance} of Proposition \ref{Prop_relations} and from point \ref{P_SN} of this Proposition.
\end{description}

%%%%%%%%%%%%%%%%%%%%%%%%%%%%%%%%%%%%%%%%%%%%%%%%%%%%%%%%%%%%%%
\vspace{0,2cm}%%%%%%%%%%%%%%%%%%%%%%%%%%%%%%%%%%%%%%%%%%%%%%%%
\noindent\textbf{Proof of Proposition \ref{group_property}}%%%
%%%%%%%%%%%%%%%%%%%%%%%%%%%%%%%%%%%%%%%%%%%%%%%%%%%%%%%%%%%%%%
%%%%%%%%%%%%%%%%%%%%%%%%%%%%%%%%%%%%%%%%%%%%%%%%%%%%%%%%%%%%%%
\vspace{0,2cm}

\noindent Let $a,b\in A$ and $i,k\in\left\{1,\ldots,n\right\}$ such that $a\succsim^{R_1,R_{2}}_{(i,k),{\cal D^{'}}}b$, $i_1\geq i$ and $k_1\leq k$ and consider the following cases:
\begin{itemize}
\item $R_{2}=N$; then $a\succsim^{R_1,R_{2}}_{(i,k),{\cal D^{'}}}b$ means that $a\succsim^{R_1}_{(i,k)}b$ for all DMs. Being $\succsim^{R_{1}}\;\subseteq\;\succsim^{R_{1}^{'}}$, we have $a\succsim^{R_1^{'}}_{(i,k)}b$ for all DMs; being  $i_1\geq i$ and $k_1\leq k$, by point \ref{nec_inclusion} (if $R_{1}=N$) or \ref{pos_inclusion} (if $R_1=P$) of Proposition \ref{Prop_relations} we have $a\succsim^{R_1^{'}}_{(i_1,k_1)}b$ for all DMs (that is $a\succsim^{R_1^{'},R_{2}^{'}}_{(i_1,k_1),{\cal D^{'}}}b$ with $R_{2}^{'}=N$) implying also $a\succsim^{R_1^{'}}_{(i_1,k_1)}b$ for at least one DM (that is $a\succsim^{R_1^{'},R_{2}^{'}}_{(i_1,k_1),{\cal D^{'}}}b$ with $R_{2}^{'}=P.$)
\item $R_{2}=P$; then $a\succsim^{R_1,R_{2}}_{(i,k),{\cal D^{'}}}b$ means that $a\succsim^{R_1}_{(i,k)}b$ for at least one DM, $d_h\in{\cal D^{'}}$. Being $\succsim^{R_{1}}\;\subseteq\;\succsim^{R_{1}^{'}}$, we have $a\succsim^{R_1^{'}}_{(i,k)}b$ for at least one DM $d_h\in{\cal D^{'}}$; being  $i_1\geq i$ and $k_1\leq k$, by point \ref{nec_inclusion} (if $R_{1}=N$) or \ref{pos_inclusion} (if $R_1=P$) of Proposition \ref{Prop_relations} we have $a\succsim^{R_1^{'}}_{(i_1,k_1)}b$ for at least one DM, $d_h\in{\cal D^{'}}$ (that is $a\succsim^{R_1^{'},R_{2}^{'}}_{(i_1,k_1),{\cal D^{'}}}b$ with $R_{2}^{'}=P$).
\end{itemize}

%%%%%%%%%%%%%%%%%%%%%%%%%%%%%%%%%%%%%%%%%%%%%%%%%%%%%%%%%%%%%%%%%
\vspace{0,2cm}%%%%%%%%%%%%%%%%%%%%%%%%%%%%%%%%%%%%%%%%%%%%%%%%%%%
\noindent\textbf{Proof of Proposition \ref{group_normal_index}}%%
%%%%%%%%%%%%%%%%%%%%%%%%%%%%%%%%%%%%%%%%%%%%%%%%%%%%%%%%%%%%%%%%%
%%%%%%%%%%%%%%%%%%%%%%%%%%%%%%%%%%%%%%%%%%%%%%%%%%%%%%%%%%%%%%%%%
\begin{enumerate}
\item It follows from point \ref{N_incl} of Proposition \ref{Prop_relations};
\item It follows from point \ref{P_incl} of Proposition \ref{Prop_relations}.
\end{enumerate}

%%%%%%%%%%%%%%%%%%%%%%%%%%%%%%%%%%%%%%%%%%%%%%%%%%%%%%%%%%%%%%
\vspace{0,2cm}%%%%%%%%%%%%%%%%%%%%%%%%%%%%%%%%%%%%%%%%%%%%%%%%
\noindent\textbf{Proof of Proposition \ref{Compl_prop}}%%%%%%%
%%%%%%%%%%%%%%%%%%%%%%%%%%%%%%%%%%%%%%%%%%%%%%%%%%%%%%%%%%%%%%
%%%%%%%%%%%%%%%%%%%%%%%%%%%%%%%%%%%%%%%%%%%%%%%%%%%%%%%%%%%%%%
\begin{enumerate}
\item Let $a,b\in A$ and $i,k\in\left\{1,\ldots,n\right\}$ such that $a\not\succsim_{(i,k),{\cal D^{'}}}^{N,N}b$. This means that there exists at least one DM, $d_{h}\in{\cal D^{'}}$, such that $a\not\succsim_{(i,k)}^{N}b$; from this, by point \ref{imprec_completeness} of Proposition \ref{Prop_relations}, we have $b\succsim_{(k,i)}^{P}a$ for at least one DM, $d_{h}\in{\cal D^{'}}$, and therefore $b\succsim_{(k,i),{\cal D^{'}}}^{P,P}a.$
\item Let $a,b\in A$ and $i,k\in\left\{1,\ldots,n\right\}$ such that $a\not\succsim_{(i,k),{\cal D^{'}}}^{N,P}b$. This means that for all DMs, $d_{h}\in{\cal D^{'}}$, $a\not\succsim_{(i,k)}^{N}
b$; from this, by point \ref{imprec_completeness} of Proposition \ref{Prop_relations}, we have $b\succsim_{(k,i)}^{P}a$ for all DMs, $d_{h}\in{\cal D^{'}}$, therefore $b\succsim_{(k,i),{\cal D^{'}}}^{P,N}a.$
\end{enumerate}

%%%%%%%%%%%%%%%%%%%%%%%%%%%%%%%%%%%%%%%%%%%%%%%%%%%%%%%%%%%%%%%%%%
\vspace{0,2cm}%%%%%%%%%%%%%%%%%%%%%%%%%%%%%%%%%%%%%%%%%%%%%%%%%%%%
\noindent\textbf{Proof of Proposition \ref{group_transitivity}}%%%
%%%%%%%%%%%%%%%%%%%%%%%%%%%%%%%%%%%%%%%%%%%%%%%%%%%%%%%%%%%%%%%%%%
%%%%%%%%%%%%%%%%%%%%%%%%%%%%%%%%%%%%%%%%%%%%%%%%%%%%%%%%%%%%%%%%%%

\vspace{0,2cm}
\noindent Let be $a,b,c\in A$, and $i,k\in\left\{1,\ldots,n\right\}$ such that $a\succsim^{N,R_1}_{(i,k),{\cal D^{'}}}b$ $b\succsim^{N,R_2}_{(i_1,k_1),{\cal D^{'}}}c$ and $k\geq i_1$. We can consider the following cases:
\begin{itemize}
\item If $R_{1}=R_{2}=N$ then $a\succsim^{N,R_1}_{(i,k),{\cal D^{'}}}b$ means that $a\succsim^{N}_{(i,k)}b$ for all DMs $d_{h}\in{\cal D^{'}}$, and $b\succsim^{N,R_2}_{(i_1,k_1),{\cal D^{'}}}c$ means that $b\succsim^{N}_{(i_1,k_1)}c$ for all DMs. By point \ref{nec_transitivity} of Proposition \ref{Prop_relations} we have $a\succsim^{N}_{(r,s)}c$ for all DMs, $d_h\in{\cal D^{'}}$ thus $a\succsim^{N,N}_{(r,s),{\cal D^{'}}}c$. 
\item If $R_{1}=N$ and $R_{2}=P$ then $a\succsim^{N,R_1}_{(i,k),{\cal D^{'}}}b$ means that $a\succsim^{N}_{(i,k)}b$ for all DMs $d_{h}\in{\cal D^{'}}$, and $b\succsim^{N,R_2}_{(i_1,k_1),{\cal D^{'}}}c$ means that $b\succsim^{N}_{(i_1,k_1)}c$ for at least one DM, $d_h\in{\cal D^{'}}$. By point \ref{nec_transitivity} of Proposition \ref{Prop_relations} we have $a\succsim^{N}_{(r,s)}c$ for at least one DM, $d_h\in{\cal D^{'}}$, thus $a\succsim^{N,P}_{(r,s),{\cal D^{'}}}c$. 
\item If $R_{1}=P$ and $R_{2}=N$ then $a\succsim^{N,R_1}_{(i,k),{\cal D^{'}}}b$ means that $a\succsim^{N}_{(i,k)}b$ for at least one DM $d_{h}\in{\cal D^{'}}$, and $b\succsim^{N,R_2}_{(i_1,k_1),{\cal D^{'}}}c$ means that $b\succsim^{N}_{(i_1,k_1)}c$ for all DMs, $d_h\in{\cal D^{'}}$. By point \ref{nec_transitivity} of Proposition \ref{Prop_relations} we have $a\succsim^{N}_{(r,s)}c$ for at least one DM, $d_h\in{\cal D^{'}}$, thus $a\succsim^{N,P}_{(r,s),{\cal D^{'}}}c$. 
\end{itemize}
\noindent The second and the third part of this Proposition can be proved analogously to the first part, using points \ref{necpos_transitivity} and \ref{posnec_transitivity} of Proposition \ref{Prop_relations}, respectively. 

\end{document}